\documentclass[10pt]{amsart}
\usepackage[margin=1in]{geometry}
\usepackage{amsmath, amsthm, amssymb, amsfonts, color, hyperref}
\hypersetup{
    colorlinks,
    citecolor=blue,
    filecolor=black,
    linkcolor=black,
    urlcolor=black
}
\usepackage[final]{showkeys}
\usepackage[disable]{todonotes}

\renewcommand{\sl}{\mathfrak{sl}}
\newcommand{\asl}{\widehat{\sl}}
\newcommand{\Tr}{\text{Tr}}
\newcommand{\hh}{\mathfrak{h}}

\newcommand{\wtilde}{\widetilde}

\newcommand{\CC}{\mathbb{C}}

\newcommand{\ZZ}{\mathbb{Z}}

\newcommand{\eps}{\varepsilon}
\newcommand{\Imm}{\text{Im}}

\newcommand{\wsl}{\wtilde{\sl}}
\newcommand{\whh}{\wtilde{\hh}}
\newcommand{\wrho}{\wtilde{\rho}}

\newcommand{\hotimes}{\widehat{\otimes}}

\newcommand{\wI}{\wtilde{I}}

\newcommand{\mult}{\text{mult}}

\newcommand{\ch}{\mathrm{h}} 

\newcommand{\wJ}{\wtilde{J}}

\newcommand{\cl}{\mathrm{cl}}
\newcommand{\Sym}{\text{Sym}}
\newcommand{\wDelta}{\wtilde{\Delta}}
\newcommand{\wk}{\kappa}
\newcommand{\mk}{\mathsf{k}}
\newcommand{\aff}{\text{aff}}

\theoremstyle{definition}
\newtheorem{thm}{Theorem}[section]
\newtheorem{prop}[thm]{Proposition}
\newtheorem{lemma}[thm]{Lemma}

\newtheorem{conj}[thm]{Conjecture}
\newtheorem*{remark}{Remark}
\numberwithin{equation}{section}

\begin{document}
\title{Affine Macdonald conjectures and special values of Felder-Varchenko functions}
\date{\today}

\author{Eric M. Rains}
\address{E.R.: Department of Mathematics\\ California Institute of Technology\\ Mathematics 253-37\\Pasadena, CA 91125, USA}
\email{rains@caltech.edu}

\author{Yi Sun}
\address{Y.S.: Department of Mathematics\\ Columbia University\\ 2990 Broadway\\ New York, NY 10027, USA}
\email{yisun@math.columbia.edu}

\author{Alexander Varchenko}
\address{A.V.: Department of Mathematics\\ University of North Carolina at Chapel Hill\\Chapel Hill, NC 27599-3250, USA}
\email{anv@email.unc.edu}

\begin{abstract}
We refine the statement of the denominator and evaluation conjectures for affine Macdonald polynomials proposed by Etingof-Kirillov~Jr. in \cite{EK3} and prove the first non-trivial cases of these conjectures.  Our results provide a $q$-deformation of the computation of genus 1 conformal blocks via elliptic Selberg integrals by Felder-Stevens-Varchenko in \cite{FSV}.  They allow us to give precise formulations for the affine Macdonald conjectures in the general case which are consistent with computer computations.

Our method applies recent work of the second named author to relate these conjectures in the case of $U_q(\asl_2)$ to evaluations of certain theta hypergeometric integrals defined by Felder-Varchenko in \cite{FV4}.  We then evaluate the resulting integrals, which may be of independent interest, by well-chosen applications of the elliptic beta integral introduced by Spiridonov in \cite{Spi00}.
\end{abstract}
\maketitle

\setcounter{tocdepth}{1}
\tableofcontents

\section{Introduction}

The present work leverages a connection between two approaches to generalizing Macdonald theory to the affine setting to provide precise statements of affine analogues of Macdonald's conjectures and proofs of the first non-trivial versions.   In \cite{EK3}, Etingof-Kirillov~Jr. defined affine Macdonald polynomials as traces of intertwiners of quantum affine algebras and stated rough analogues of Macdonald's conjectures.  In \cite{FV4}, Felder-Varchenko defined their so-called elliptic Macdonald polynomials in the $U_q(\asl_2)$ case in terms of theta hypergeometric integrals and conjectured that they coincided with Etingof-Kirillov~Jr.'s affine Macdonald polynomials.  In the recent work \cite{Sun:qafftr}, the second named author resolved this conjecture and provided a precise connection between these two objects.

In this paper, we use this connection to refine and correct the statements of the affine Macdonald denominator and evaluation conjectures in \cite{EK3} and prove the first cases of these conjectures.  More precisely, in the case of $U_q(\asl_2)$, we use the results of \cite{Sun:qafftr} to express the affine Macdonald denominator and evaluation in terms of special values of certain theta hypergeometric integrals related to Felder-Varchenko functions.  We then give evaluations of these integrals (which may be of independent interest) by manipulating the integrands to match instances of the elliptic beta integral introduced by Spiridonov in \cite{Spi00}.  Combined with computer computations, our results allow us to conjecture precise formulations for the affine Macdonald conjectures in the general case; our formulation includes an additional prefactor whose limit is consistent with a new term appearing in the affine Gindikin-Karpelevich formula for $p$-adic loop groups recently studied in \cite{BFK12, BK13, BGKP14, BKP16}.

Our work is motivated by two different streams of literature.  First, our main results provide a $q$-deformation of the first non-trivial case of the computations of conformal blocks via elliptic Selberg integrals given by Felder-Stevens-Varchenko in \cite[Theorem 5.1]{FSV}.  While our method is intrinsically different, we anticipate that it will generalize to give $q$-deformations of the other conformal blocks present in \cite[Theorem 5.1]{FSV}.  Second, in future work we plan to approach the general case of affine denominator and evaluation conjectures stated here using quantum affine algebras in the spirit of \cite{EK-mac}.  We intend the present work as a guide for finding the correct statements.

In the remainder of this introduction, we state our results in detail and give additional motivation and background.  In this introduction, we will mix additive and multiplicative notations, where in all cases the appropriate choice is clear from context.  For the reader's convenience, all notations will be reintroduced in full detail in later sections, and in each section we will specify whether additive or multiplicative notations are used.

\subsection{Affine Macdonald polynomials}

Fix an integer $\mk \geq 0$. Letting $L_{\mu + k\Lambda_0}$ denote the irreducible integrable module for $U_q(\asl_{n})$ and $V$ the fundamental representation of $U_q(\sl_{n})$, for a dominant integral weight $\mu + k \Lambda_0$ and $v \in \Sym^{n (\mk - 1)}V[0]$, there is a unique intertwiner 
\[
\Upsilon_{\mu, k, \mk}^v(z): L_{\mu + k \Lambda_0 + (\mk - 1)\wrho} \to L_{\mu + k \Lambda_0 + (\mk - 1)\wrho} \hotimes \Sym^{n(\mk - 1)}V(z)
\]
such that $\Upsilon^v_{\mu, k, \mk}(z) v_{\mu + k \Lambda_0 + (\mk - 1)\wrho} = v_{\mu + k \Lambda_0 + (\mk - 1)\wrho} \otimes v + (\text{l.o.t.})$, where $(\text{l.o.t.})$ denotes terms of lower weight in the first tensor factor.  Fixing a choice of $w_0 \in \Sym^{n(\mk - 1)}V [0]$ and making the identification $\Sym^{n(\mk - 1)}V [0] \simeq \CC \cdot w_0$, define the trace function
\[
\chi_{\mu, k, \mk}(q, \lambda, \omega) = \Tr|_{L_{\mu + k \Lambda_0 + (\mk - 1)\wrho}}\Big(\Upsilon^{w_0}_{\mu, k, \mk}(z) q^{2\lambda + 2\omega d}\Big).
\]
In \cite{EK3}, by analogy with their approach to ordinary Macdonald polynomials in \cite{EK}, Etingof-Kirillov~Jr. defined the affine Macdonald polynomial for $\asl_n$ at $t = q^{\mk}$ as
\[
J_{\mu, k, \mk}(q, \lambda, \omega) := \frac{\chi_{\mu, k, \mk}(q, \lambda, \omega)}{\chi_{0, 0, \mk}(q, \lambda, \omega)}.
\]

\subsection{Statement of the main results}

In \cite{EK3}, Etingof-Kirillov~Jr. state analogues of the Macdonald denominator and evaluation conjectures in the affine setting.  We extend and correct their conjectures in the following two conjectures, which are a refinement of \cite[Theorem 11.1]{EK3} and a correction to \cite[Conjecture 11.3]{EK3}, respectively.

{\renewcommand{\thethm}{\ref{conj:f-val-conj}}
\begin{conj}[Affine denominator conjecture] 
The affine Macdonald denominator is given by 
\[
\chi_{0, 0, \mk}(q, \lambda, \omega) = q^{2(\mk - 1)(\rho, \lambda)} \frac{\prod_{i = 1}^{\mk - 1} (q^{-2\omega + 2i}; q^{-2\omega})}{\prod_{i = 1}^{\mk - 1} (q^{-2\omega + 2ni}; q^{-2\omega})}  \prod_{i = 1}^{\mk - 1} \prod_{\alpha > 0} (1 - q^{-2(\alpha, \lambda +\omega d) + 2i})^{\mult(\alpha)}.
\]
\end{conj}
}

{\renewcommand{\thethm}{\ref{conj:aff-eval}}
\begin{conj}[Affine evaluation conjecture]
For $|q| > 1$, we have that
\[
J_{\mu, k, \mk}(q, \mk \rho, \mk n) = q^{2(\mu, \mk \rho)} \frac{\prod_{i = 1}^{\mk - 1} (q^{-2i}; q^{-2(k + \mk n)})}{\prod_{i = 1}^{\mk - 1} (q^{-2ni}; q^{-2(k + \mk n)})} \frac{\prod_{i = 1}^{\mk - 1} (q^{-2ni}; q^{-2\mk n})}{\prod_{i = 1}^{\mk - 1} (q^{-2i}; q^{-2\mk n})}  \prod_{\alpha > 0} \prod_{i = 0}^{\mk - 1} \frac{(1 - q^{-2(\alpha, \mu + k \Lambda_0 + \mk \wrho) - 2i})^{\mult(\alpha)}}{(1 - q^{-2(\alpha, \mk \wrho) - 2i})^{\mult(\alpha)}}.
\]
\end{conj}
}

Our main results are the following first nontrivial cases of Conjectures \ref{conj:f-val-conj} and \ref{conj:aff-eval}.  For $n = 2$ and $\mk = 2$, Theorem \ref{thm:f-val} provides a computation of the multiplicative correction factor $f(p, q)$ for the affine Macdonald denominator in \cite[Theorem 11.1]{EK3}, and Theorem \ref{thm:eval-conj} corrects and proves the affine Macdonald evaluation conjecture of \cite[Conjecture 11.3]{EK3}.  

{\renewcommand{\thethm}{\ref{thm:f-val}}
\begin{thm} 
For $n = 2$ and $\mk = 2$, the affine denominator is given by
\[
\chi_{0, 0, 2}(q, \lambda, \omega) = q^\lambda \frac{(q^{-2\omega + 2}; q^{-2\omega})}{(q^{-2\omega + 4}, q^{-2\omega})} (q^{-2\lambda + 2}; q^{-2\omega}) (q^{2\lambda + 2} q^{-2\omega}; q^{-2\omega}) (q^{-2\omega + 2}; q^{-2\omega}).
\]
\end{thm}
}

{\renewcommand{\thethm}{\ref{thm:eval-conj}}
\begin{thm}
For $n = 2$, $\mk = 2$, and $|q| > 1$, the affine Macdonald polynomial satisfies the evaluation
\[
J_{\mu, k, 2}(q, 2, 4) = q^{2\mu} \frac{(q^{-2}; q^{-2\wk})}{(q^{-4}; q^{-2\wk})} \frac{\theta_0(q^{-2\mu - 4}; q^{-2\wk}) (q^{-2\mu - 6}; q^{-2\wk})(q^{2\mu + 2} q^{-2\wk}; q^{-2\wk})(q^{-2\wk}; q^{-2\wk}) (q^{-2\wk - 2}; q^{-2\wk})}{(q^{-4}; q^{-2})(q^{-6};q^{-8}) (q^{-2}; q^{-8})}.
\]
\end{thm}
}

\subsection{Felder-Varchenko functions and elliptic Macdonald polynomials}

Our method proceeds by using the results of \cite{Sun:qafftr} to connect the trace functions from Theorems \ref{thm:f-val} and \ref{thm:eval-conj} to theta hypergeometric integrals defined by Felder-Varchenko in \cite{FV4}. In \cite{FV4}, for a positive integer level $\kappa \geq 4$, $\mu \neq \pm 1 \pmod{\kappa}$, and $\Imm(\eta) < 0$, Felder-Varchenko defined the non-symmetric hypergeometric theta function of level $\kappa + 2$ by 
\[
\wDelta_{\mu, \kappa}(\lambda; \tau, \eta) = e^{\frac{2\pi i \eta}{\kappa} \mu^2} Q(2\eta\mu; -2\eta\kappa, \eta) \wI_{\mu, \kappa}(\lambda; \tau, \eta),
\]
where
\[
Q(\mu; \sigma, \eta) := \frac{\theta(4\eta; \sigma) \theta'(0; \sigma)}{\theta(\mu - 2\eta; \sigma) \theta(\mu + 2\eta; \sigma)}
\]
and
\begin{multline*}
\wI_{\mu, \kappa}(\lambda; \tau, \eta) := e^{\pi i \tau \frac{\mu^2}{2\kappa} - \pi i \lambda \mu} (2 \kappa \tau; 2\kappa \tau) \int_\gamma \frac{\Gamma(t + 2\eta; \tau, -2\eta\kappa)}{\Gamma(t - 2\eta; \tau, -2\eta \kappa)}\\
 \frac{\theta(t + \lambda; \tau)}{\theta(t - 2\eta; \tau)} \frac{\theta(t + 2\eta \mu; -2\eta \kappa)}{\theta(t - 2\eta; -2\eta\kappa)} \theta_0(1/2 + \mu \tau + \kappa \tau - \kappa \lambda + 2t; 2\kappa\tau) dt,
\end{multline*}
where the cycle $\gamma$ travels from $-1/2$ to $1/2$, lies above the pole at $t = -2\eta$ and below the pole at $t = 2\eta$, and separates all other poles above and below the real line.  They defined further the symmetrized hypergeometric theta function by 
\[
\Delta_{\mu, \kappa}(\lambda; \tau, \eta) := \wDelta_{\mu, \kappa}(\lambda; \tau, \eta) - \wDelta_{\mu, \kappa}(-\lambda; \tau, \eta).
\]
and the elliptic Macdonald polynomial for $t = q^2$ by 
\[
P_{\mu, \kappa}(\lambda; \tau, \eta) := e^{-\pi i \frac{4\eta + \tau}{2\kappa}(\mu + 2)^2 + \pi i 3\tau/4} \frac{\Delta_{\mu + 2, \kappa}(\lambda; \tau, \eta)}{\theta(\lambda - 2\eta; \tau)\theta(\lambda; \tau)\theta(\lambda + 2\eta; \tau)}.
\]
In \cite{Sun:qafftr}, the elliptic and affine Macdonald polynomials were related in the following result. 

{\renewcommand{\thethm}{\ref{prop:fv-conj}}
\begin{prop}[{\cite[Theorem 9.9]{Sun:qafftr}}] 
For $|q| > 1$, $|q^{-2\omega}| < |q^{-6}|$, and $q^{-2\mu}$ sufficiently close to $0$, the elliptic and affine Macdonald polynomials for $U_q(\asl_2)$ are related by
\begin{multline*}
J_{\mu, k, 2}(q, \lambda, \omega)\\
 = \frac{P_{\mu, \wk}(2\eta\lambda; -2\eta\omega, \eta)}{2\pi f_{2, 2}(q, q^{-2\omega})} \frac{(q^{-4}; q^{-2\omega})(q^{-2\omega}; q^{-2\omega})^3}{(q^{-2\omega + 2}; q^{-2\omega})} \frac{(q^{-2\omega + 2}; q^{-2\omega}, q^{-2\wk})^2}{(q^{-2\omega - 2}; q^{-2\omega}, q^{-2\wk})^2} \frac{q^{\mu + 4}(q^{-2\mu - 6}; q^{-2\wk})(q^{2\mu + 2} q^{-2\wk}; q^{-2\wk})}{(q^{-4}; q^{-2\wk})(q^{-2\wk}; q^{-2\wk})},
\end{multline*}
where $f_{2, 2}(q, q^{-2\omega})$ is the normalizing function of Proposition \ref{prop:ek-const-2} and $\wk = k + 4$.
\end{prop}
}

By applying Proposition \ref{prop:fv-conj}, we are able to deduce Theorems \ref{thm:f-val} and \ref{thm:eval-conj} from Theorems \ref{thm:ellmac-val} and \ref{thm:ellmac-eval} on special values of the elliptic Macdonald polynomials from \cite{FV4}.  These special values are $q$-deformations of the first cases of \cite[Theorem 5.1]{FSV}.

{\renewcommand{\thethm}{\ref{thm:ellmac-val}}
\begin{thm}
We have that 
\[
P_{0, 4}(\lambda; \tau, \eta) = - 2\pi \frac{\Gamma(-6\eta; \tau, -8\eta)}{\Gamma(-2\eta; \tau, -8\eta)} \frac{1}{\theta_0(4\eta; \tau)(\tau;\tau)^3} \frac{(-4\eta; -4\eta)}{(-2\eta; -4\eta)}.
\]
\end{thm}
}
{\renewcommand{\thethm}{\ref{thm:ellmac-eval}}
\begin{thm} 
We have that 
\[
P_{\mu, \kappa}(4\eta; -8\eta, \eta) = -2 \pi e^{-12\pi i \eta - 2\pi i (\mu + 2)\eta} \frac{\Gamma(-6\eta; -2\kappa\eta, -8\eta)}{\Gamma(-2\eta; -2\kappa\eta, -8\eta)} \frac{\theta_0(2(\mu + 2)\eta; -2\kappa\eta) (-2\kappa \eta; -2\kappa\eta)^2}{(-8\eta; -8\eta) (-4\eta; -4\eta)^2 (-2\eta; -2\eta)}.
\]
\end{thm}
}

\subsection{Theta hypergeometric integral evaluations}

Our proofs of Theorems \ref{thm:ellmac-val} and \ref{thm:ellmac-eval} are based on a single theta hypergeometric integral evaluation, which may be of independent interest.  For modular parameters $\tau, \eta$ with $\Imm(\tau), \Imm(\eta) > 0$, consider the theta hypergeometric integral
\[
\wI(\lambda; \tau, \eta) := e^{-3\pi i \lambda} \int_\gamma \frac{\Gamma(t - 2\eta; \tau, 8\eta)}{\Gamma(t + 2 \eta; \tau, 8\eta)} \frac{\theta_0(t + \lambda; \tau)}{\theta_0(t + 2\eta; \tau)} \frac{\theta_0(t - 4\eta; 8\eta)}{\theta_0(t + 2\eta; 8\eta)} \theta_0(2t + 6\tau - 4\lambda + 1/2; 8\tau) dt,
\]
where the cycle $\gamma$ travels from $-1/2$ to $1/2$, lies above the pole at $t = -2 \eta$ and below the pole at $t = 2 \eta$, and separates all other poles above and below the real line.  Define the symmetrization
\[
I(\lambda; \tau, \eta) := \wI(\lambda; \tau, \eta) - \wI(-\lambda; \tau, \eta).
\]
The core technical tool of this paper consists of Theorem \ref{thm:eval3}, which gives an explicit evaluation for this integral based on manipulations of theta functions and well-chosen applications of the elliptic beta integral of \cite{Spi00}.
{\renewcommand{\thethm}{\ref{thm:eval3}}
\begin{thm}
We have the expression
\[
I(\lambda; \tau, \eta) = e^{-12\pi i \eta} \frac{\Gamma(6\eta; \tau, 8\eta)}{\Gamma(2\eta; \tau, 8\eta)}\frac{1}{(8\tau; 8\tau)\theta_0(- 4\eta; \tau)}  \frac{1}{(4\eta; 4\eta) (2\eta + 1/2; 2\eta)}  e^{-3\pi i \lambda} \theta_0(\lambda; \tau) \theta_0(\lambda - 2\eta; \tau) \theta_0(\lambda + 2\eta; \tau).
\]
\end{thm}
}

\subsection{Organization of paper}

The remainder of this paper is organized as follows.  In Section \ref{sec:int-eval}, we prove the theta hypergeometric integral evaluation of Theorem \ref{thm:eval3} as well as two easier evaluations involving Felder-Varchenko functions.  In Section \ref{sec:fv}, we explain Felder-Varchenko functions and elliptic Macdonald polynomials and prove Theorems \ref{thm:fv-val1}, \ref{thm:fv-val2}, \ref{thm:ellmac-val}, and \ref{thm:ellmac-eval} giving certain special values for them using the integral evaluations in Section \ref{sec:int-eval}.  In Section \ref{sec:aff-mac}, we use Proposition \ref{prop:fv-conj} to prove Theorems \ref{thm:f-val} and \ref{thm:eval-conj} on the first cases of the affine Macdonald denominator and evaluation conjectures.  We then combine them with evidence from computer computations to state Conjecture \ref{conj:f-val-conj} on the affine Macdonald denominator and Conjecture \ref{conj:aff-eval} on the affine Macdonald evaluation.  We conclude by discussing our conjectures in the classical, affine Hall, and critical limits.

\subsection{Acknowledgments}

Y.S. and A.V. thank the Max-Planck-Institut f\"ur Mathematik in Bonn for providing excellent working conditions.  Y.S. thanks P. Etingof for many helpful discussions.  E.M.R. was partially supported by NSF grant DMS-1500806.  This work was partially supported by a Junior Fellow award from the Simons Foundation to Yi Sun.  A.V. was partially supported by NSF grant DMS-1362924 and Simons Foundation grant \#336826.

\section{Three elliptic hypergeometric integral evaluations} \label{sec:int-eval}

In this section, we present three evaluations of elliptic hypergeometric integrals which arise in the study of Felder-Varchenko functions.  Our method rests upon applications of the elliptic beta integral of \cite{Spi00} after rearrangement of the integrand.

\subsection{Notations and elliptic functions}

We now give our conventions and notations for $q$-Pochhammer symbols, theta functions, and elliptic gamma functions.  Because our techniques connect two streams of work which use both additive and multiplicative notations in a critical way, we will abuse notation and use both notations; in all cases, which one is meant will be evident from the context.  We have labeled which of multiplicative or additive notations are used in this section, but we will omit these in the main text to avoid obscuring the notation.  The remainder of Section \ref{sec:int-eval} after this subsection will only use additive notation.

\subsubsection{Theta functions}

We use the single and double $q$-Pochhammer symbols, denoted in multiplicative notation by 
\[
(u; q)_\text{mult} := \prod_{n \geq 0} (1 - u q^n) \qquad \text{ and } \qquad (u; q, r)_\text{mult} := \prod_{n, m \geq 0} (1 - u q^n r^m)
\]
for $|q|, |r| < 1$ and in additive notation by
\[
(z; \tau)_\text{add} := \prod_{n \geq 0} (1 - e^{2\pi i z + 2\pi i \tau n}) \qquad \text{ and } \qquad (z; \tau, \sigma)_\text{add} := \prod_{n, m \geq 0} (1 - q^{2\pi i z + 2\pi i \tau n + 2\pi i \sigma m})
\]
for $\Imm(\tau), \Imm(\sigma) > 0$.  Define the theta function in multiplicative and additive notation by
\[
\theta_0^\text{mult}(u; q) := (u; q)(u^{-1}q; q) \qquad \text{ and } \qquad \theta_0^\text{add}(z; \tau) := (z; \tau)(\tau - z; \tau).
\]
Define Jacobi's first theta function, which will appear only in additive notation, by
\[
\theta(z; \tau) := i e^{\pi i \tau/4 - \pi i z}(\tau; \tau) \theta_0(z; \tau).
\]
It satisfies the modular relation
\begin{equation} \label{eq:theta-mod}
\theta(z/\tau; -1/\tau) = -i \sqrt{-i\tau} e^{\frac{\pi i z^2}{\tau}} \theta(z; \tau),
\end{equation}
where the square root takes values in the right half plane.

\subsubsection{Elliptic gamma functions}

Define the elliptic gamma function in multiplicative and additive notation by
\[
\Gamma^\text{mult}(u; q, r) := \frac{(u^{-1} qr; q, r)}{(u; q, r)} \qquad \text{ and } \qquad \Gamma^\text{add}(z; \tau, \sigma) = \frac{(\tau + \sigma - z; \tau, \sigma)}{(z; \tau, \sigma)}.
\]
In \cite{FV3}, it was shown to have the following three-term modular relation of $SL(3, \ZZ)$-type.

\begin{prop}[{\cite[Theorem 4.1]{FV3}}] \label{prop:sl3-mod}
The elliptic gamma function satisfies the modular relation
\begin{equation} \label{eq:ellgam-mod}
\Gamma(z/\sigma; \tau/\sigma, -1/\sigma) = e^{\pi i Q(z; \tau, \sigma)} \Gamma((z - \sigma)/\tau; -1/\tau, -\sigma/\tau) \Gamma(z; \tau, \sigma),
\end{equation}
where 
\[
Q(z; \tau, \sigma) = \frac{z^3}{3\tau \sigma} - \frac{\tau + \sigma - 1}{2 \tau \sigma} z^2 + \frac{\tau^2 + \sigma^2 + 3 \tau \sigma - 3 \tau - 3 \sigma + 1}{6 \tau \sigma} z + \frac{1}{12}(\tau + \sigma - 1)(\tau^{-1} + \sigma^{-1} - 1).
\]
\end{prop}

\subsubsection{Product notation}

We will often use the presence of multiple arguments before the modular parameter to indicate a product of multiple factors.  For example, we have that
\[
\Gamma(\pm z; \tau, \sigma) := \Gamma(z; \tau, \sigma) \Gamma(-z; \tau, \sigma)
\]
and that 
\[
\theta_0(u^{\pm}, v^{\pm}; q) := \theta_0(u; q) \theta_0(u^{-1}; q) \theta_0(v; q) \theta_0(v^{-1}; q).
\]

\subsection{Evaluations of the first kind}

The first evaluations we consider are related to Felder-Varchenko functions specialized at particular values.  They take the following form.

\begin{thm} \label{thm:eval1}
We have that 
\[
\int_\gamma \frac{\Gamma(t + 1/4; \tau, \sigma)}{\Gamma(t - 1/4; \tau, \sigma)} \frac{\theta_0(t + 1/2; \tau)}{\theta_0(t - 1/4; \tau)} \frac{\theta_0(t + 1/2; \sigma)}{\theta_0(t - 1/4; \sigma)} dt = -(1 + i) \frac{\Gamma(1/4; \tau, \sigma)}{\Gamma(3/4; \tau, \sigma)} \frac{1}{(\tau; \tau)(\tau+1/2; 2\tau)} \frac{1}{(\sigma; \sigma)(\sigma+1/2; 2\sigma)},
\]
where the cycle $\gamma$ travels from $-\frac{1}{2}$ to $\frac{1}{2}$, lies above the pole at $t = -1/4$ and below the pole at $t = 1/4$, and separates all other poles above and below the real axis.
\end{thm}

\begin{thm} \label{thm:eval2}
We have that 
\[
\int_\gamma \frac{\Gamma(t - 1/4; \tau, \sigma)}{\Gamma(t + 1/4; \tau, \sigma)} \frac{\theta_0(t + 1/2; \tau)}{\theta_0(t + 1/4; \tau)} \frac{\theta_0(t + 1/2; \sigma)}{\theta_0(t + 1/4; \sigma)} dt = - (1 - i) \frac{\Gamma(3/4; \tau, \sigma)}{\Gamma(1/4; \tau, \sigma)} \frac{1}{(\tau; \tau)(\tau + 1/2; 2\tau)} \frac{1}{(\sigma; \sigma) (\sigma + 1/2; 2\sigma)},
\]
where the cycle $\gamma$ travels from $-\frac{1}{2}$ to $\frac{1}{2}$, lies above the pole at $t = 1/4$ and below the pole at $t = -1/4$, and separates all other poles above and below the real axis.
\end{thm}

\subsection{Evaluations of the second kind}

The second type of evaluation we consider is related to a special value of the elliptic Macdonald polynomial.  For modular parameters $\tau, \eta$ with $\Imm(\tau), \Imm(\eta) > 0$, we consider the theta hypergeometric integral
\begin{equation} \label{eq:asym-int}
\wI(\lambda; \tau, \eta) := e^{-3\pi i \lambda} \int_\gamma \frac{\Gamma(t - 2\eta; \tau, 8\eta)}{\Gamma(t + 2 \eta; \tau, 8\eta)} \frac{\theta_0(t + \lambda; \tau)}{\theta_0(t + 2\eta; \tau)} \frac{\theta_0(t - 4\eta; 8\eta)}{\theta_0(t + 2\eta; 8\eta)} \theta_0(2t + 6\tau - 4\lambda + 1/2; 8\tau) dt,
\end{equation}
where the cycle $\gamma$ travels from $-\frac{1}{2}$ to $\frac{1}{2}$, lies above the pole at $t = -2 \eta$ and below the pole at $t = 2\eta$, and separates all other poles above and below the real axis.  Define the symmetrization
\[
I(\lambda; \tau, \eta) := \wI(\lambda; \tau, \eta) - \wI(-\lambda; \tau, \eta).
\]
We will give an explicit evaluation of $I(\lambda; \tau, \eta)$.

\begin{thm} \label{thm:eval3}
We have the expression
\[
I(\lambda; \tau, \eta) = e^{-12\pi i \eta} \frac{\Gamma(6\eta; \tau, 8\eta)}{\Gamma(2\eta; \tau, 8\eta)}\frac{1}{(8\tau; 8\tau)\theta_0(- 4\eta; \tau)}  \frac{1}{(4\eta; 4\eta) (2\eta + 1/2; 2\eta)}  e^{-3\pi i \lambda} \theta_0(\lambda; \tau) \theta_0(\lambda - 2\eta; \tau) \theta_0(\lambda + 2\eta; \tau).
\]
\end{thm}

\subsection{Proof of evaluations of the first kind}

In this subsection we prove Theorems \ref{thm:eval1} and \ref{thm:eval2}. We apply the elliptic beta integral introduced by Spiridonov in \cite{Spi00}, formulated in our notations as follows.

\begin{thm}[{\cite[Theorem 1]{Spi00}}] \label{thm:ell-beta}
Let $\tau, \sigma$ be modular parameters with $\Imm(\tau), \Imm(\sigma) > 0$, and let $s_1, \ldots, s_6$ be parameters with $\Imm(s_i) > 0$ so that $\sum_{i = 1}^6 s_i = \tau + \sigma$.  Then we have
\[
\int_{\gamma} \frac{\prod_{i = 1}^6 \Gamma(\pm t + s_i; \tau, \sigma)}{\Gamma(\pm 2t; \tau, \sigma)} dt = \frac{2 \prod_{i < j} \Gamma(s_i + s_j; \tau, \sigma)}{(\tau; \tau) (\sigma; \sigma)},
\]
where the contour $\gamma$ is a line from $-1/2$ to $1/2$.
\end{thm}

We proceed by simplifying the integrands into the form of the elliptic beta integral for modular parameters $\tau/2$ and $\sigma$.   

\begin{proof}[Proof of Theorem \ref{thm:eval1}]
Denote the original integrand by $J(t)$ and define
\[
I(t) := \frac{\Gamma(2t - 1/4; \tau, \sigma)}{\Gamma(2t + 1/4; \tau, \sigma)} \frac{\theta_0(2t; \tau)}{\theta_0(2t + 1/4; \tau)} \frac{\theta_0(2t; \sigma)}{\theta_0(2t + 1/4; \sigma)}
\]
so that $I(t) = J(2t + 1/2)$.  We note that $\int_\gamma J(t) dt = \int_\gamma I(t) dt$ since $I(t)$ is $1$-periodic in $t$.  Denoting $\Gamma(z) := \Gamma(z; \tau/2, \sigma)$, observe now that 
\begin{multline*}
\frac{\Gamma(2t - 1/4; \tau, \sigma)}{\Gamma(2t + 1/4; \tau, \sigma)} \frac{1}{\theta_0(2t + 1/4; \tau)} \frac{1}{\theta_0(2t + 1/4; \sigma)} = \Gamma(\pm 2t - 1/4; \tau, \sigma) \frac{\theta_0(\tau + 2t + 1/4; \tau)}{\theta_0(2t + 1/4; \tau)} \\ = i e^{-4\pi i t}\Gamma(\pm 2t - 1/4; \tau, \sigma)
 = i e^{-4\pi i t} \Gamma(\pm t - 1/8) \Gamma(\pm t + 3/8) \Gamma(\pm t + \sigma/2 - 1/8)\Gamma(\pm t + \sigma/2 + 3/8)
\end{multline*}
and that 
\begin{multline*}
\theta_0(2t; \tau) \theta_0(2t; \sigma) = \frac{\theta_0(2t; \tau)}{\Gamma(2t) \Gamma(\sigma - 2t)} = \frac{1}{\Gamma(\pm 2t)} \frac{\theta_0(2t; \tau)}{\theta_0(\tau/2 + 2t; \tau/2)} \\= - e^{4\pi i t} \frac{1}{\Gamma(\pm 2t)} \frac{1}{\theta_0(2t + \tau/2; \tau)} = - e^{4\pi i t} \frac{\Gamma(\pm t + \tau/4)\Gamma(\pm t + \tau/4 + 1/2)}{\Gamma(\pm 2t)}.
\end{multline*}
Substituting these in, we find that 
\[
I(t) = - i \frac{\Gamma(\pm t - 1/8) \Gamma(\pm t + 3/8) \Gamma(\pm t + \sigma/2 - 1/8) \Gamma(\pm t + \sigma/2 + 3/8) \Gamma(\pm t + \tau/4) \Gamma(\pm t + \tau/4 + 1/2)}{\Gamma(\pm 2t)}.
\]
Therefore, $I(t)$ corresponds to the integrand of the elliptic beta integral with 
\[
(s_1, \ldots, s_6) = (-1/8, 3/8, \sigma/2 - 1/8, \sigma/2 + 3/8, \tau/4, \tau/4 + 1/2)
\]
and modular parameters $\tau/2$ and $\sigma$.  By Theorem \ref{thm:ell-beta}, we conclude that 
\begin{align*}
\int_\gamma I(t) dt &= - \frac{2i}{(\tau/2; \tau/2)(\sigma; \sigma)} \Gamma(1/4)\Gamma(\sigma/2 - 1/4) \Gamma(\sigma/2 + 1/4) \Gamma(\tau/4 - 1/8)\Gamma(\tau/4 + 3/8)\\
&\phantom{==} \Gamma(\sigma/2 + 1/4) \Gamma(\sigma/2 + 3/4) \Gamma(\tau/4 + 3/8) \Gamma(\tau/4 + 7/8) \Gamma(\sigma + 1/4)\Gamma(\sigma/2 + \tau/4 - 1/8)\\
&\phantom{==} \Gamma(\sigma/2 + \tau/4 + 3/8) \Gamma(\sigma/2 + \tau/4 + 3/8)\Gamma(\sigma/2 + \tau/4 + 7/8) \Gamma(\tau/2 + 1/2)\\
&= - \frac{2i}{(\tau/2; \tau/2)(\sigma; \sigma)} \Gamma(\sigma + 1/2; \tau, 2\sigma)^2 \Gamma(\tau/2 - 1/4; \tau, 2\sigma)^2\Gamma(\sigma + \tau/2 - 1/4; \tau, 2\sigma)^2\\
&\phantom{==} \Gamma(1/4)\Gamma(\sigma + 1/4) \Gamma(\tau/2 + 1/2)\\
&= - \frac{2i}{(\tau/2; \tau/2)(\sigma; \sigma)} \Gamma(\sigma + 1/2; \tau, 2\sigma)^2 \Gamma(\tau/2 - 1/4; \tau, \sigma)^2 \Gamma(1/4)\Gamma(\sigma + 1/4) \Gamma(\tau/2 + 1/2),
\end{align*}
where we apply the duplication formulas for the elliptic gamma function of \cite{FV:mult}.  Notice now the identities
\begin{align*}
\Gamma(\tau/2 + 1/2) &= \frac{(\sigma + 1/2; \tau/2, \sigma)}{(\tau/2 + 1/2; \tau/2, \sigma)} = \frac{(1/2; \sigma)}{(1/2; \tau/2)}\\
\Gamma(\sigma + 1/2; \tau, 2\sigma) &= \frac{(\tau + \sigma + 1/2; \tau, 2\sigma)}{(\sigma + 1/2; \tau, 2\sigma)} = \frac{1}{(\sigma + 1/2; 2\sigma)}\\
\Gamma(\tau/2 - 1/4; \tau, \sigma)^2 \Gamma(1/4)\Gamma(\sigma + 1/4) &= \frac{(\tau/2 + \sigma + 1/4; \tau, \sigma)^2}{(\tau/2 - 1/4; \tau, \sigma)^2} \frac{(\tau/2 - 1/4; \tau/2, \sigma)(\tau/2 + \sigma - 1/4; \tau/2, \sigma)}{(1/4; \tau/2, \sigma)(\sigma + 1/4; \tau/2, \sigma)}\\
&= \frac{(\tau - 1/4; \tau, \sigma)^2}{(\sigma + 1/4; \tau, \sigma)^2} \frac{1}{\theta_0(1/4; \tau/2)}.
\end{align*}
Substituting these into the previous expression yields
\begin{align*}
\int_\gamma I(t) dt &= - 2i \Gamma(\sigma + 1/4; \tau, \sigma)^2 \frac{1}{(\tau/2; \tau/2)(1/2; \tau/2) \theta_0(1/4; \tau/2)} \frac{(1/2; \sigma)}{(\sigma; \sigma)(\sigma + 1/2; 2\sigma)^2}\\
&= -i(1+i) \frac{\Gamma(1/4; \tau, \sigma)}{\Gamma(3/4; \tau, \sigma)} \frac{\theta_0(1/4; \tau)}{(\tau; \tau) (1/2; \tau)} \frac{(1/2; \sigma)}{(\sigma; \sigma) (\sigma+1/2; 2\sigma)^2 \theta_0(3/4; \sigma)}.
\end{align*}
Observe now that
\[
\frac{\theta_0(1/4; \tau)}{(\tau; \tau) (1/2; \tau)} = \frac{1}{1 + i} \frac{(1/4; \tau)(-1/4; \tau)}{(\tau; \tau)(1/2;\tau)}= \frac{1}{1 + i} \frac{(1/2; 2\tau)}{(\tau; \tau)(1/2;\tau)} = \frac{1}{1 + i} \frac{1}{(\tau; \tau)(\tau+1/2; 2\tau)}
\]
and
\begin{multline*}
\frac{(1/2; \sigma)}{(\sigma; \sigma) (\sigma+1/2; 2\sigma)^2 \theta_0(3/4; \sigma)} = (1 - i) \frac{(1/2; \sigma)}{(\sigma; \sigma)(\sigma + 1/2; 2\sigma)^2 (3/4; \sigma)(1/4; \sigma)}\\ = (1 - i) \frac{(1/2; 2\sigma)}{(\sigma; \sigma)(\sigma + 1/2; 2\sigma) (1/2; 2\sigma)} = (1 - i) \frac{1}{(\sigma; \sigma)(\sigma + 1/2; 2\sigma)}.
\end{multline*}
Substituting in, we conclude as desired that
\[
\int_\gamma I(t) dt = -(1 + i) \frac{\Gamma(1/4; \tau, \sigma)}{\Gamma(3/4; \tau, \sigma)} \frac{1}{(\tau; \tau)(\tau+1/2; 2\tau)} \frac{1}{(\sigma; \sigma)(\sigma+1/2; 2\sigma)}. \qedhere
\]
\end{proof}

\begin{proof}[Proof of Theorem \ref{thm:eval2}]
As in the proof of Theorem \ref{thm:eval1}, denote the original integrand by $J(t)$ and define
\[
I(t) := \frac{\Gamma(2t + 1/4; \tau, \sigma)}{\Gamma(2t - 1/4; \tau, \sigma)} \frac{\theta_0(2t; \tau)}{\theta_0(2t - 1/4; \tau)} \frac{\theta_0(2t; \sigma)}{\theta_0(2t - 1/4; \sigma)}
\]
so that $I(t) = J(2t + 1/2)$ and hence that $\int_\gamma J(t) dt = \int_\gamma I(t) dt$ since $I(t)$ is $1$-periodic in $t$.  Denoting $\Gamma(z) := \Gamma(z; \tau/2, \sigma)$, observe now that 
\begin{multline*}
\frac{\Gamma(2t + 1/4; \tau, \sigma)}{\Gamma(2t - 1/4; \tau, \sigma)} \frac{1}{\theta_0(2t - 1/4; \tau)} \frac{1}{\theta_0(2t - 1/4; \sigma)} = \Gamma(\pm 2t + 1/4; \tau, \sigma) \frac{\theta_0(\tau + 2t - 1/4; \tau)}{\theta_0(2t - 1/4; \tau)} \\ = -i e^{-4\pi i t}\Gamma(\pm 2t + 1/4; \tau, \sigma)
 = -i e^{-4\pi i t} \Gamma(\pm t + 1/8) \Gamma(\pm t + 5/8) \Gamma(\pm t + \sigma/2 + 1/8)\Gamma(\pm t + \sigma/2 + 5/8)
\end{multline*}
and that 
\begin{multline*}
\theta_0(2t; \tau) \theta_0(2t; \sigma) = \frac{\theta_0(2t; \tau)}{\Gamma(2t) \Gamma(\sigma - 2t)} = \frac{1}{\Gamma(\pm 2t)} \frac{\theta_0(2t; \tau)}{\theta_0(\tau/2 + 2t; \tau/2)}\\ =- e^{4\pi i t} \frac{1}{\Gamma(\pm 2t)} \frac{1}{\theta_0(2t + \tau/2; \tau)} = - e^{4\pi i t} \frac{\Gamma(\pm t + \tau/4)\Gamma(\pm t + \tau/4 + 1/2)}{\Gamma(\pm 2t)}.
\end{multline*}
We conclude that 
\[
I(t) = i \frac{\Gamma(\pm t + 1/8)\Gamma(\pm t + 5/8)\Gamma(\pm t + \sigma/2 + 1/8)\Gamma(\pm t + \sigma/2 + 5/8)\Gamma(\pm t + \tau/4) \Gamma(\pm t + \tau/4 + 1/2)}{\Gamma(\pm 2t)},
\]
which is the integrand of the elliptic beta integral with modular parameters $\tau/2$ and $\sigma$ and 
\[
(s_1, \dots, s_6) = (1/8, 5/8, \sigma/2 +1/8, \sigma/2 + 5/8, \tau/4, \tau/4 + 1/2).
\]
By Theorem \ref{thm:ell-beta}, we conclude that 
\begin{align*}
\int_\gamma I(t) dt &= \frac{2i}{(\tau/2; \tau/2)(\sigma; \sigma)} \Gamma(3/4)\Gamma(\sigma/2+1/4)\Gamma(\sigma/2+3/4) \Gamma(\tau/4+1/8)\Gamma(\tau/4+5/8) \\
&\phantom{==} \Gamma(\sigma/2+3/4) \Gamma(\sigma/2 + 1/4) \Gamma(\tau/4 +5/8) \Gamma(\tau/4 + 1/8) \Gamma(\sigma + 3/4) \Gamma(\sigma/2+\tau/4+1/8)\\ 
&\phantom{==} \Gamma(\sigma/2+\tau/4+5/8)\Gamma(\sigma/2+\tau/4+5/8)\Gamma(\sigma/2+\tau/4+1/8)\Gamma(\tau/2 +1/2)\\
&= \frac{2i}{(\tau/2; \tau/2)(\sigma; \sigma)} \Gamma(\sigma+1/2; \tau, 2 \sigma)^2 \Gamma(\tau/2+1/4; \tau, 2\sigma)^2 \Gamma(\sigma + \tau/2 + 1/4; \tau, 2\sigma)^2 \\
&\phantom{==} \Gamma(3/4) \Gamma(\sigma + 3/4) \Gamma(\tau/2 +1/2)\\
&= \frac{2i}{(\tau/2; \tau/2)(\sigma; \sigma)} \Gamma(\sigma+1/2; \tau, 2 \sigma)^2  \Gamma(\tau/2+1/4; \tau, \sigma)^2 \Gamma(3/4) \Gamma(\sigma + 3/4) \Gamma(\tau/2 +1/2),
\end{align*}
where we apply the duplication formulas for the elliptic gamma function of \cite{FV:mult}.  We observe now the identities
\begin{align*}
\Gamma(\tau/2 + 1/2) &= \frac{(\sigma + 1/2; \tau/2, \sigma)}{(\tau/2 + 1/2; \tau/2, \sigma)} = \frac{(1/2; \sigma)}{(1/2; \tau/2)}\\
\Gamma(\sigma + 1/2; \tau, 2\sigma) &= \frac{(\tau + \sigma + 1/2; \tau, 2\sigma)}{(\sigma + 1/2; \tau, 2\sigma)} = \frac{1}{(\sigma + 1/2; 2\sigma)}\\
\Gamma(\tau/2+1/4; \tau, \sigma)^2 \Gamma(3/4) \Gamma(\sigma + 3/4) &= \frac{(\sigma + \tau/2 - 1/4; \tau, \sigma)^2}{(\tau/2 + 1/4; \tau, \sigma)^2} \frac{(\tau/2 + \sigma + 1/4; \tau/2, \sigma) (\tau/2 + 1/4; \tau/2, \sigma)}{(3/4; \tau/2, \sigma)(\sigma + 3/4; \tau/2, \sigma)}\\
&= \frac{(\tau + 1/4; \tau, \sigma)^2}{(\sigma - 1/4; \tau, \sigma)^2} \frac{1}{(\tau/2 + 1/4; \tau/2) (3/4; \tau/2)}.
\end{align*}
Substituting this into the previous result yields
\begin{align*}
\int_\gamma I(t) dt &= 2i \frac{\Gamma(\sigma + 3/4; \tau, \sigma)}{\Gamma(\tau + 1/4; \tau, \sigma)} \frac{1}{(1/2; \tau/2) (\tau/2;\tau/2) \theta_0(3/4; \tau/2)} \frac{(1/2; \sigma)}{(\sigma; \sigma) (\sigma + 1/2; 2\sigma)^2}\\
&= i(1 - i) \frac{\Gamma(3/4; \tau, \sigma)}{\Gamma(1/4; \tau, \sigma)} \frac{\theta_0(3/4; \tau)}{(\tau; \tau)(1/2; \tau)} \frac{(1/2; 2\sigma)}{(\sigma; \sigma) (\sigma + 1/2; 2\sigma) \theta_0(1/4; \sigma)}\\
&= - (1 - i) \frac{\Gamma(3/4; \tau, \sigma)}{\Gamma(1/4; \tau, \sigma)} \frac{1}{(\tau; \tau)(\tau + 1/2; 2\tau)} \frac{1}{(\sigma; \sigma) (\sigma + 1/2; 2\sigma)}. \qedhere
\end{align*}
\end{proof}

\subsection{Proof of the evaluation of the second kind}

The remainder of this section is devoted to the proof of Theorem \ref{thm:eval3}.  Our strategy will be to show that the value of the integral does not change after symmetrization of the integrand and then to evaluate this simpler symmetrized integrand.  We defer some computations with theta functions to Subsection \ref{sec:theta-comp} and some integral evaluations resulting from the elliptic beta integral to Subsection \ref{sec:int-eval-sec}.  Define the intermediate expressions
\[
\wJ_1(t, \tau, \eta) := \Gamma(t - 2 \eta; \tau, 8\eta) \Gamma(-t - 2\eta; \tau, 8 \eta)\theta_0(t + 4 \eta; 8\eta)
\]
and
\[
\wJ_2(t, \lambda, \tau) := e^{-3\pi i \lambda} \theta_0(t + \lambda; \tau) \theta_0(2t + 6 \tau - 4 \lambda + 1/2; 8\tau).
\]
Notice that $\wJ_1(t, \tau, \eta) = \wJ_1(-t, \tau, \eta)$.  We use this decomposition to show in Lemmas \ref{lem:int-rearrange}, \ref{lem:theta-simp}, and \ref{lem:full-sym} that a symmetrization of the integrand admits a simpler expression.  First, we show that $\wJ_1$ and $\wJ_2$ give a factorization of the integrand into symmetric and non-symmetric parts.

\begin{lemma} \label{lem:int-rearrange}
We have that 
\[
\wI(\lambda, \tau, \eta) = e^{-12\pi i \eta} \int_\gamma \wJ_1(t, \tau, \eta) \wJ_2(t, \lambda, \tau) dt.
\]
\end{lemma}
\begin{proof}
Applying Lemma \ref{lem:sym-rearrange}, we find that
\begin{align*}
\wI(\lambda; \tau, \eta) &= -e^{-3\pi i \lambda - 4 \pi i \eta} \int_\gamma e^{-2\pi i t}\Gamma(t - 2\eta; \tau, 8\eta)\Gamma(-t - 2\eta; \tau, 8\eta) \theta_0(t + \lambda; \tau) \theta_0(t - 4 \eta; 8\eta) \theta_0(2t + 6\tau - 4\lambda + 1/2; 8\tau) dt\\
&= e^{-3\pi i \lambda - 12 \pi i \eta} \int_\gamma \Gamma(t - 2\eta; \tau, 8\eta)\Gamma(-t - 2\eta; \tau, 8\eta) \theta_0(t + 4 \eta; 8\eta) \theta_0(t + \lambda; \tau)  \theta_0(2t + 6\tau - 4\lambda + 1/2; 8\tau) dt.
\end{align*}
We conclude as desired that 
\begin{align*}
\wI(\lambda, \tau, \eta) &= \int_\gamma e^{-12 \pi i \eta - 3\pi i \lambda} \Gamma(t - 2\eta; \tau, 8\eta)\Gamma(-t - 2\eta; \tau, 8\eta) \theta_0(t + 4 \eta; 8\eta) \theta_0(t + \lambda; \tau)\theta_0(2t + 6\tau - 4\lambda + 1/2; 8\tau) dt\\
&= e^{-12\pi i \eta} \int_\gamma \wJ_1(t, \tau, \eta) \wJ_2(t, \lambda, \tau) dt. \qedhere
\end{align*}
\end{proof}

\begin{lemma} \label{lem:theta-simp}
We have that 
\[
\wJ_2(t, \lambda, \tau) - \wJ_2(-t, -\lambda,\tau) = \frac{2\theta_0(6\tau + 1/2; 8\tau)e^{\pi i \lambda - 2 \pi i t} \theta_0(t + \lambda; \tau)\theta_0(t - 2\lambda + 1/2; 2 \tau)}{\theta_0(1/2; 2\tau)}.
\]
\end{lemma}
\begin{proof}
Notice first that
\begin{align*}
\wJ_2(-t, -\lambda, \tau) &= e^{3\pi i \lambda} \theta_0(-t - \lambda; \tau) \theta_0(-2t + 6\tau + 4 \lambda + 1/2; 8\tau)\\
&= e^{3\pi i \lambda} \theta_0(\tau + t + \lambda; \tau) \theta_0(2t + 2\tau - 4 \lambda + 1/2; 8\tau)\\
&= - e^{\pi i \lambda - 2 \pi i t} \theta_0(t + \lambda; \tau) \theta_0(2t + 2\tau - 4 \lambda + 1/2; 8\tau).
\end{align*}
We conclude that 
\[
\wJ_2(t, \lambda, \tau) - \wJ_2(-t, -\lambda,\tau) = e^{-3\pi i \lambda} \theta_0(t + \lambda; \tau) \Big(\theta_0(2t + 6\tau - 4\lambda + 1/2; 8\tau) + e^{4 \pi i \lambda - 2\pi i t} \theta_0(2t + 2 \tau - 4 \lambda + 1/2; 8\tau)\Big).
\]
By Lemma \ref{lem:theta-simp2} applied with $z = t - 2 \lambda$ and $\sigma = 2 \tau$, we find that 
\begin{multline*}
\theta_0(2t + 6\tau - 4\lambda + 1/2; 8\tau) + e^{4 \pi i \lambda - 2\pi i t} \theta_0(2t + 2 \tau - 4 \lambda + 1/2; 8\tau) \\
= \frac{2\theta_0(6\tau + 1/2; 8\tau) e^{-2\pi i t + 4 \pi i \lambda} \theta_0(t - 2\lambda + 1/2; 2 \tau)}{\theta_0(1/2; 2\tau)}.
\end{multline*}
We conclude as desired that
\[
\wJ_2(t, \lambda, \tau) - \wJ_2(-t, -\lambda,\tau) = \frac{2\theta_0(6\tau + 1/2; 8\tau)e^{\pi i \lambda - 2 \pi i t} \theta_0(t + \lambda; \tau)\theta_0(t - 2\lambda + 1/2; 2 \tau)}{\theta_0(1/2; 2\tau)}. \qedhere
\]
\end{proof}

\begin{lemma} \label{lem:full-sym}
We have that
\begin{multline*}
\wJ_2(t, \lambda, \tau) - \wJ_2(t, -\lambda, \tau) + \wJ_2(-t, \lambda, \tau) - \wJ_2(-t, -\lambda, \tau) \\
= \frac{4\theta_0(6\tau +1/2; 8\tau) \theta_0(\lambda; \tau)}{\theta_0(1/2;2\tau)^3} e^{-3\pi i \lambda}\Big(\frac{\theta_0(2\lambda + 1/2;2\tau)}{\theta_0(\tau + 1/2;2\tau)} e^{-2\pi i t}\theta_0(t+\tau + 1/2; 2\tau) \theta_0(t+1/2;2\tau)^2\\ - \frac{\theta_0(\lambda+1/2;\tau)^2}{\theta_0(\tau; 2\tau)}e^{-2\pi i t}\theta_0(t+\tau + 1/2; 2\tau) \theta_0(t; 2\tau)^2\Big).
\end{multline*}
\end{lemma}
\begin{proof}
Denote the quantity on the left by $h(t, \lambda)$. Applying Lemma \ref{lem:theta-simp} twice, we conclude that
\begin{multline*}
h(t, \lambda) = \frac{2\theta_0(6\tau + 1/2; 8\tau)}{\theta_0(1/2; 2\tau)} e^{-2\pi i t}\\ \Big(e^{\pi i \lambda} \theta_0(t + \lambda; \tau) \theta_0(t - 2\lambda + 1/2; 2\tau) - e^{-\pi i \lambda} \theta_0(t - \lambda; \tau) \theta_0(t + 2 \lambda + 1/2; 2\tau)\Big).
\end{multline*}
Substituting in the result of Lemma \ref{lem:theta-simp3} yields the desired result.
\end{proof}

We are now ready to prove Theorem \ref{thm:eval3} by relating the value of the integral to its symmetrized version.

\begin{proof}[Proof of Theorem \ref{thm:eval3}]
Since the cycle $\gamma$ is invariant under the change of variables $t \mapsto -t$, we see that 
\[
\int_\gamma \wJ_1(t, \tau, \eta) \wJ_2(t, \lambda, \tau) dt = \int_\gamma \wJ_1(-t, \tau, \eta) \wJ_2(-t, \lambda, \tau) dt = \int_\gamma \wJ_1(t, \tau, \eta) \wJ_2(-t, \lambda, \tau) dt.
\]
By Lemma \ref{lem:int-rearrange}, this implies that 
\[
I(\lambda, \tau, \eta) = \frac{e^{-12\pi i \eta}}{2} \int_\gamma \wJ_1(t, \tau, \eta) \Big(\wJ_2(t, \lambda, \tau) - \wJ_2(t, -\lambda, \tau) + \wJ_2(-t, \lambda, \tau) - \wJ_2(-t, -\lambda, \tau)\Big) dt.
\]
Define the integral evaluations
\begin{align*}
I_1 &:= \int_\gamma \wJ_1(t, \tau, \eta) e^{-2\pi i t} \theta_0(t + \tau + 1/2; 2\tau) \theta_0(t + 1/2; 2 \tau)^2 dt\\
I_2 &:= \int_\gamma \wJ_1(t, \tau, \eta) e^{-2\pi i t} \theta_0(t + \tau + 1/2; 2\tau) \theta_0(t; 2 \tau)^2 dt
\end{align*}
so that by Lemma \ref{lem:full-sym} we have that 
\begin{equation} \label{eq:temp-form}
I(\lambda, \tau, \eta) = \frac{2e^{-12\pi i \eta - 3\pi i \lambda} \theta_0(6\tau + 1/2;8\tau) \theta_0(\lambda; \tau)}{\theta_0(1/2; 2\tau)^3} \Big(\frac{\theta_0(2\lambda + 1/2; 2\tau)}{\theta_0(\tau+1/2; 2\tau)} I_1 - \frac{\theta_0(\lambda + 1/2; \tau)^2}{\theta_0(\tau; 2\tau)} I_2\Big).
\end{equation}
By the integral evaluations of Lemmas \ref{lem:int-eval1} and \ref{lem:int-eval2}, we find with $\Gamma(z) := \Gamma(z; 2\tau, 8\eta)$ that
\begin{align*}
I_1&=\frac{2}{(2\tau; 2\tau)(8\eta; 8\eta)} \Gamma(-4\eta + \tau)\Gamma(6\eta + 1/2)\Gamma(-2\eta) \Gamma(2\eta +1/2)\Gamma(-2\eta + \tau)^2 \Gamma(-2\eta + 2\tau)\\
&\phantom{==} \Gamma(8\eta + 1/2)\Gamma(12\eta)\Gamma(8\eta + \tau + 1/2) \Gamma(4\eta + 1/2)\Gamma(\tau)\\
I_2&=-\frac{2}{(2\tau; 2\tau)(8\eta; 8\eta)} \Gamma(-4\eta + \tau)\Gamma(6\eta)\Gamma(-2\eta + 1/2) \Gamma(2\eta +1/2)\Gamma(-2\eta + \tau)\\
&\phantom{==} \Gamma(6\eta + \tau)\Gamma(-2\eta + \tau + 1/2)\Gamma(2\eta + \tau + 1/2)\Gamma(-2\eta + 2\tau)\\
&\phantom{==} \Gamma(8\eta + 1/2)\Gamma(12\eta +1/2)\Gamma(8\eta + \tau) \Gamma(4\eta)\Gamma(\tau+1/2).
\end{align*}
Having now expressed the integral as an explicit theta function in (\ref{eq:temp-form}), it remains only to simplify this theta function.  Define the quantity
\[
I_0(\lambda, \tau, \eta) := \frac{\theta_0(2\lambda+1/2; 2\tau)}{\theta_0(\tau + 1/2; 2\tau)} I_1 - \frac{\theta_0(\lambda + 1/2; \tau)^2}{\theta_0(\tau; 2\tau)} I_2
\]
so that 
\begin{equation} \label{eq:i-0-rel}
I(\lambda, \tau, \eta) = \frac{2e^{-12\pi i \eta - 3\pi i \lambda} \theta_0(6\tau + 1/2;8\tau) \theta_0(\lambda; \tau)}{\theta_0(1/2; 2\tau)^3} I_0(\lambda, \tau, \eta).
\end{equation}
We claim now that $I_0(2\eta, \tau, \eta) = 0$; indeed, we observe that
\[
\frac{I_1}{I_2} = - \frac{\theta_0(4\eta; 2\tau) \theta_0(\tau + 1/2; 2\tau) \theta_0(-2\eta + 1/2; 2\tau)}{\theta_0(-2\eta; 2\tau) \theta_0(4\eta +1/2; 2\tau) \theta_0(\tau; 2\tau) \theta_0(-2\eta + \tau; 2\tau)} \frac{1}{\Gamma(-2\eta + \tau + 1/2)\Gamma(2\eta + \tau + 1/2)}
\]
where
\[
\frac{\theta_0(4\eta; 2\tau)}{\theta_0(-2\eta; 2\tau) \theta_0(-2\eta + \tau; 2\tau)} = \frac{\theta_0(4\eta; 2\tau)}{\theta_0(-2\eta; \tau)} = - e^{4\pi i \eta} \theta_0(2\eta + 1/2; \tau)
\]
and 
\begin{multline*}
\Gamma(-2\eta + \tau + 1/2)\Gamma(2\eta + \tau + 1/2) = \frac{(\tau + 10\eta + 1/2; 2\tau, 8\eta)(\tau + 6\eta + 1/2; 2\tau, 8\eta)}{(\tau - 2 \eta + 1/2; 2 \tau, 8\eta)(\tau + 2\eta + 1/2; 2\tau, 8\eta)}\\ = \frac{1}{(\tau - 2\eta + 1/2; 2\tau)(\tau + 2\eta + 1/2; 2\tau, 8\eta)} = \frac{1}{\theta_0(\tau + 2\eta + 1/2; 2\tau)}
\end{multline*}
and
\[
\theta_0(-2\eta + 1/2; 2\tau) \theta_0(\tau + 2\eta + 1/2; 2\tau) = e^{-4\pi i \eta} \theta_0(2\eta + 1/2; \tau).
\]
Substituting in, we conclude that 
\[
\frac{I_1}{I_2} = \frac{\theta_0(\theta + 1/2; 2\tau) \theta_0(2\eta + 1/2; \tau)^2}{\theta_0(4\eta + 1/2; 2\tau) \theta_0(\tau; 2\tau)}
\]
and therefore that $I_0(2\eta, \tau, \eta) = 0$.  Notice that $I(\lambda, \tau, \eta) = I(-\lambda, \tau, \eta)$ by definition, so by (\ref{eq:i-0-rel}) we conclude that
\[
I_0(-\lambda, \tau, \eta) = -e^{-4\pi i \lambda} I_0(\lambda, \tau, \eta).
\]
and therefore that $I'(2\eta, \tau, \eta) = I'(-2\eta, \tau, \eta) = 0$.  Notice that $I'(\lambda, \tau, \eta)$ is a theta function in $\lambda$ with period $\tau$ and multiplier $e^{-4\pi i \lambda}$, hence we have that
\[
I'(\lambda, \tau, \eta) = C(\tau, \eta) \theta_0(\lambda - 2 \eta; \tau) \theta_0(\lambda + 2 \eta; \tau)
\]
for some function $C(\tau, \eta)$.  To compute the value of $C(\tau, \eta)$, we set $\lambda = 1/2$ and apply the alternate expression for $I_1$ in Lemma \ref{lem:theta-simp4} to obtain
\begin{align*}
C(\tau, \eta) &= \frac{I_0(1/2, \tau, \eta)}{\theta_0(-2\eta + 1/2; \tau)\theta_0(2\eta + 1/2; \tau)}\\
&= \frac{\theta_0(1/2; 2\tau)}{\theta_0(\tau + 1/2; 2\tau) \theta_0(-2\eta+1/2; \tau)\theta_0(2\eta+1/2; \tau)} I_1\\
&= 2 \frac{\Gamma(6\eta; \tau, 8\eta)}{\Gamma(2\eta; \tau, 8\eta)} \frac{\theta_0(1/2; 2\tau)(\tau + 1/2; \tau)\theta_0(2\eta+1/2; \tau) \theta_0(\tau + 2\eta + 1/2; \tau)}{\theta_0(\tau + 1/2; 2\tau) \theta_0(-2\eta+1/2; \tau)\theta_0(2\eta+1/2; \tau)(\tau; \tau) \theta_0(\tau + 4\eta; \tau)} \frac{1}{(4\eta; 4\eta) (2\eta + 1/2; 2\eta)}\\
&= 4 \frac{\Gamma(6\eta; \tau, 8\eta)}{\Gamma(2\eta; \tau, 8\eta)} \frac{(2 \tau + 1/2; 2\tau)^2(\tau + 1/2; \tau)}{(\tau + 1/2; 2\tau)^2(\tau; \tau) \theta_0(\tau + 4\eta; \tau)} \frac{1}{(4\eta; 4\eta) (2\eta + 1/2; 2\eta)}\\
&= 4 \frac{\Gamma(6\eta; \tau, 8\eta)}{\Gamma(2\eta; \tau, 8\eta)} \frac{(2 \tau + 1/2; 2\tau)^3}{(\tau + 1/2; 2\tau)(\tau; \tau)\theta_0(\tau + 4\eta; \tau)} \frac{1}{(4\eta; 4\eta) (2\eta + 1/2; 2\eta)}.
\end{align*}
We deduce now that 
\begin{align*}
I(\lambda, \tau, \eta) &= \frac{2 e^{-12\pi i \eta}\theta_0(6\tau + 1/2; 8\tau) }{\theta_0(1/2; 2\tau)^3}  C(\tau, \eta) e^{-3\pi i \lambda} \theta_0(\lambda; \tau) \theta_0(\lambda - 2\eta; \tau) \theta_0(\lambda + 2\eta; \tau)\\
&= e^{-12\pi i \eta}\frac{\Gamma(6\eta; \tau, 8\eta)}{\Gamma(2\eta; \tau, 8\eta)}\frac{(2\tau + 1/2; 4\tau)}{(2\tau + 1/2; 2\tau)^3(\tau + 1/2; 2\tau)(\tau; \tau)\theta_0(- 4\eta; \tau)}  \frac{1}{(4\eta; 4\eta) (2\eta + 1/2; 2\eta)}\\
&\phantom{==}  e^{-3\pi i \lambda} \theta_0(\lambda; \tau) \theta_0(\lambda - 2\eta; \tau) \theta_0(\lambda + 2\eta; \tau).
\end{align*}
Notice now that 
\begin{multline*}
\frac{(2\tau + 1/2; 4\tau)}{(2\tau + 1/2; 2\tau)^3(\tau + 1/2; 2\tau)(\tau; \tau)} = \frac{(2\tau + 1/2; 4\tau)}{(2\tau + 1/2; 2\tau)^2 (\tau + 1/2; \tau)(\tau; \tau)}
= \frac{(2\tau + 1/2; 4\tau)}{(2\tau + 1/2; 2\tau)^2 (2\tau; 2\tau)}\\
= \frac{(2\tau + 1/2; 4\tau)}{(2\tau + 1/2; 2\tau) (4\tau; 4\tau)}
= \frac{1}{(4\tau + 1/2; 4\tau) (4\tau; 4\tau)}
= \frac{1}{(8\tau; 8\tau)}.
\end{multline*}
We conclude as desired that 
\begin{multline*}
I(\lambda, \tau, \eta) = e^{-12\pi i \eta}\frac{\Gamma(6\eta; \tau, 8\eta)}{\Gamma(2\eta; \tau, 8\eta)}\frac{1}{(8\tau; 8\tau)\theta_0(- 4\eta; \tau)}\\  \frac{1}{(4\eta; 4\eta) (2\eta + 1/2; 2\eta)}  e^{-3\pi i \lambda} \theta_0(\lambda; \tau) \theta_0(\lambda - 2\eta; \tau) \theta_0(\lambda + 2\eta; \tau). \qedhere
\end{multline*}
\end{proof}

\subsection{Computations with theta functions} \label{sec:theta-comp}

In this section we perform some computations with theta functions which are used in the previous subsections.

\begin{lemma} \label{lem:sym-rearrange}
We have the identity
\[
\frac{\Gamma(t - 2\eta; \tau, 8\eta)}{\Gamma(t + 2 \eta; \tau, 8\eta)} \frac{1}{\theta_0(t + 2\eta; \tau)} \frac{1}{\theta_0(t + 2\eta; 8\eta)} = -e^{-2\pi i t - 4 \pi i \eta} \Gamma(t - 2\eta; \tau, 8\eta)\Gamma(-t - 2\eta; \tau, 8\eta).
\]
\end{lemma}
\begin{proof}
Applying the identity $\Gamma(t; \tau, \sigma)^{-1} = \Gamma(\tau + \sigma - t; \tau, \sigma)$, we obtain
\begin{align*}
\frac{\Gamma(t - 2\eta; \tau, 8\eta)}{\Gamma(t + 2 \eta; \tau, 8\eta)} \frac{1}{\theta_0(t + 2\eta; \tau)} \frac{1}{\theta_0(t + 2\eta; 8\eta)}
&= \Gamma(t - 2\eta; \tau, 8\eta)\Gamma(-t + 6\eta + \tau; \tau, 8\eta) \frac{1}{\theta_0(t + 2\eta; \tau)} \frac{1}{\theta_0(t + 2\eta; 8\eta)}\\
&= \Gamma(t - 2\eta; \tau, 8\eta)\Gamma(-t - 2\eta; \tau, 8\eta) \frac{\theta_0(- t - 2\eta; \tau)}{\theta_0(t + 2\eta; \tau)} \frac{\theta_0(-t + 6\eta; 8\eta)}{\theta_0(t + 2\eta; 8\eta)}\\
&= -e^{-2\pi i t - 4 \pi i \eta} \Gamma(t - 2\eta; \tau, 8\eta)\Gamma(-t - 2\eta; \tau, 8\eta). \qedhere
\end{align*}
\end{proof}

\begin{lemma} \label{lem:theta-simp2}
For any modular parameter $\sigma$ with $\Imm(\sigma) > 0$, we have
\[
\theta_0(2z + 3\sigma + 1/2; 4\sigma) + e^{-2\pi i z} \theta_0(2z + \sigma + 1/2; 4\sigma) = \frac{2\theta_0(3\sigma + 1/2; 4 \sigma) e^{-2\pi i z} \theta_0(z + 1/2; \sigma)}{\theta_0(1/2; \sigma)}.
\]
\end{lemma}
\begin{proof}
Denote the expression on the left by $g(z)$ and notice that 
\begin{multline*}
g(z + \sigma) = \theta_0(2z + 5 \sigma + 1/2; 4\sigma) + e^{-2\pi i z - 2 \pi i \sigma} \theta_0(2z + 3 \sigma + 1/2; 4\sigma)\\ = e^{-4\pi i z - 2\pi i \sigma} \theta_0(2z + \sigma + 1/2; 4\sigma) + e^{-2\pi i z - 2 \pi i \sigma} \theta_0(2z + 3 \sigma + 1/2; 4\sigma) = e^{-2\pi i z - 2 \pi i \sigma} g(z).
\end{multline*}
On the other hand, we see that
\[
g(1/2) = \theta_0(3\sigma + 1/2; 4\sigma) - \theta_0(\sigma + 1/2; 4\sigma) = 0,
\]
hence $g(z)$ is a degree $1$ theta function with period $\sigma$, multiplier $e^{-2\pi i z - 2\pi i \sigma}$ and zero at $z = 1/2$.  Therefore, it is given by $C e^{-2\pi i z} \theta_0(z + 1/2; \sigma)$ for some constant of proportionality $C$.  To determine the constant $C$, substitute $z = 0$ to obtain that 
\[
C = \frac{\theta_0(3\sigma + 1/2; 4 \sigma) + \theta_0(\sigma + 1/2; 4\sigma)}{\theta_0(1/2; \sigma)} = \frac{2\theta_0(3\sigma + 1/2; 4 \sigma)}{\theta_0(1/2; \sigma)}. \qedhere
\]
\end{proof}

\begin{lemma} \label{lem:theta-simp3}
We have that 
\begin{multline*}
e^{\pi i \lambda} \theta_0(t + \lambda; \tau) \theta_0(t - 2\lambda + 1/2; 2\tau) - e^{-\pi i \lambda} \theta_0(t - \lambda; \tau) \theta_0(t + 2 \lambda + 1/2; 2\tau)\\
= 2e^{-3\pi i \lambda} \frac{\theta_0(t + \tau + 1/2; 2\tau) \theta_0(\lambda; \tau)}{\theta_0(1/2;2\tau)^2} \Big(\frac{\theta_0(2\lambda + 1/2;2\tau)}{\theta_0(\tau + 1/2;2\tau)} \theta_0(t+1/2;2\tau)^2 - \frac{\theta_0(\lambda+1/2;\tau)^2}{\theta_0(\tau; 2\tau)} \theta_0(t; 2\tau)^2\Big).
\end{multline*}
\end{lemma}
\begin{proof}
Denote the quantity on the left by $f(t)$.  Notice that $f(t + 2 \tau) = e^{-6\pi i t - 2\pi i \tau} f(t)$ and that
\begin{align*}
f(\tau + 1/2) &= e^{\pi i \lambda} \theta_0(\tau + \lambda + 1/2; \tau) \theta_0(\tau - 2\lambda; 2\tau) - e^{-\pi i \lambda} \theta_0(\tau - \lambda + 1/2; \tau)\theta_0(\tau + 2\lambda; 2\tau)\\
&= e^{-\pi i \lambda} \theta_0(\lambda + 1/2; \tau) \theta_0(\tau + 2\lambda; 2\tau) - e^{-\pi i \lambda} \theta_0(\tau - \lambda + 1/2; \tau)\theta_0(\tau + 2\lambda; 2\tau)\\
&= 0.
\end{align*}
Therefore, we see that $f(t)$ is a theta function with period $2\tau$, multiplier $e^{-6\pi i t - 2\pi \tau}$, and a zero at $t = \tau + 1/2$.  Consider
\[
g(t) = \frac{f(t)}{\theta_0(t + \tau + 1/2; 2\tau)},
\]
which is a theta function with period $2\tau$ and multiplier $e^{-4\pi i t}$.  We may write
\[
g(t) = A \theta_0(t; 2\tau)^2 + B \theta_0(t + 1/2; 2\tau)^2
\]
for some constants $A$ and $B$.  Setting $t = 1/2$, we find that 
\begin{align*}
A = \frac{g(1/2)}{\theta_0(1/2; 2\tau)^2}
&= \frac{e^{\pi i \lambda} \theta_0(\lambda + 1/2; \tau) \theta_0(-2\lambda; 2\tau) - e^{-\pi i \lambda} \theta_0(-\lambda + 1/2; \tau) \theta_0(2\lambda; 2\tau)}{\theta_0(\tau; 2\tau) \theta_0(1/2;2\tau)^2}\\
&= \frac{-2e^{-3\pi i \lambda} \theta_0(\lambda + 1/2; \tau)^2 \theta_0(\lambda; \tau)}{\theta_0(\tau; 2\tau) \theta_0(1/2;2\tau)^2},
\end{align*}
and setting $t = 0$, we find that 
\begin{align*}
B = \frac{g(0)}{\theta_0(1/2; 2\tau)^2}
&= \frac{e^{\pi i \lambda} \theta_0(\lambda; \tau) \theta_0(-2\lambda + 1/2; 2\tau) - e^{-\pi i \lambda} \theta_0(-\lambda; \tau) \theta_0(2\lambda + 1/2; 2\tau)}{\theta_0(\tau + 1/2; 2\tau) \theta_0(1/2; 2\tau)^2}\\
&= \frac{2 e^{-3\pi i \lambda} \theta_0(\lambda; \tau)\theta_0(2\lambda + 1/2; 2\tau)}{\theta_0(\tau + 1/2; 2\tau) \theta_0(1/2; 2\tau)^2}.
\end{align*}
Putting these together yields the desired result.
\end{proof}

\begin{lemma} \label{lem:theta-simp4}
We have that 
\begin{multline*}
\frac{2}{(2\tau; 2\tau)(8\eta; 8\eta)} \Gamma(-4\eta + \tau)\Gamma(6\eta + 1/2)\Gamma(-2\eta) \Gamma(2\eta +1/2)\Gamma(-2\eta + \tau)^2 \Gamma(-2\eta + 2\tau)\\
\Gamma(8\eta + 1/2)\Gamma(12\eta)\Gamma(8\eta + \tau + 1/2) \Gamma(4\eta + 1/2)\Gamma(\tau)\\
= 2 \frac{\Gamma(6\eta; \tau, 8\eta)}{\Gamma(2\eta; \tau, 8\eta)} \frac{(\tau + 1/2; \tau)\theta_0(2\eta+1/2; \tau) \theta_0(\tau + 2\eta + 1/2; \tau)}{(\tau; \tau) \theta_0(\tau + 4\eta; \tau)} \frac{1}{(4\eta; 4\eta) (2\eta + 1/2; 2\eta)}.
\end{multline*}
\end{lemma}
\begin{proof}
Denote the given expression by $S$.  We may simplify $S$ as
\begin{align*}
S&= \frac{2}{(2\tau; 2\tau)(8\eta; 8\eta)} \frac{\Gamma(6\eta + 1/2) \Gamma(2\eta + 1/2) \Gamma(8\eta + 1/2) \Gamma(12\eta) \Gamma(4\eta + 1/2) \Gamma(\tau)}{\Gamma(\tau + 12\eta) \Gamma(10\eta + 2\tau) \Gamma(10 \eta + \tau)^2 \Gamma(\tau + 1/2)\Gamma(10\eta)}\\
&= \frac{2}{(2\tau; 2\tau)(8\eta; 8\eta)} \frac{\Gamma(2\eta+1/2; 2\tau; 2\eta)}{\Gamma(10\eta; \tau, 8\eta) \Gamma(10\eta+\tau; \tau, 8\eta)} \frac{\Gamma(\tau) \Gamma(12\eta)}{\Gamma(\tau+1/2)\Gamma(\tau + 12\eta)}.
\end{align*}
Notice now the identities
\begin{align*}
\Gamma(2\eta + 1/2; 2\tau, 2\eta) &= \frac{(2\tau + 1/2; 2\tau, 2\eta)}{(2\eta+1/2; 2\tau, 2\eta)} = \frac{(2\tau + 1/2; 2\tau)}{(2\eta+1/2; 2\eta)}\\
\Gamma(\tau) &= (\tau; 2\tau)^{-1}\\
\Gamma(12\eta) &= \frac{\theta_0(4\eta; 2\tau)}{(4\eta; 8\eta)}\\
\Gamma(\tau + 1/2) &= (\tau + 1/2; 2\tau)^{-1}\\
\Gamma(\tau + 12\eta) &= \theta_0(\tau + 4\eta; 2\tau).
\end{align*}
Substituting in, we find that 
\begin{align*}
S &= 2 \frac{\Gamma(-2\eta; \tau, 8\eta)}{\Gamma(2\eta; \tau, 8\eta)} \frac{(2\tau + 1/2; 2\tau)\theta_0(4\eta; 2\tau)(\tau + 1/2; 2\tau)}{(2\tau; 2\tau)(\tau; 2\tau) \theta_0(\tau + 4\eta; 2\tau) \theta_0(2\eta; \tau)} \frac{1}{(8\eta; 8\eta) (2\eta + 1/2; 2\eta)(4\eta; 8\eta)}\\
&= 2 \frac{\Gamma(6\eta; \tau, 8\eta)}{\Gamma(2\eta; \tau, 8\eta)} \frac{(\tau + 1/2; \tau)\theta_0(2\eta+1/2; \tau)}{(\tau; \tau) \theta_0(\tau + 4\eta; 2\tau) \theta_0(-2\eta; \tau)} \frac{1}{(4\eta; 4\eta) (2\eta + 1/2; 2\eta)}\\
&= 2 \frac{\Gamma(6\eta; \tau, 8\eta)}{\Gamma(2\eta; \tau, 8\eta)} \frac{(\tau + 1/2; \tau)\theta_0(2\eta+1/2; \tau) \theta_0(\tau + 2\eta + 1/2; \tau)}{(\tau; \tau) \theta_0(\tau + 4\eta; \tau)} \frac{1}{(4\eta; 4\eta) (2\eta + 1/2; 2\eta)}.\qedhere
\end{align*}

\end{proof}

\subsection{Applications of the elliptic beta integral} \label{sec:int-eval-sec}

In this section we evaluate two integrals in Lemmas \ref{lem:int-eval1} and \ref{lem:int-eval2} using the elliptic beta integral of Theorem \ref{thm:ell-beta} introduced by Spiridonov in \cite{Spi00}. It is not obvious that these integrals may be evaluated by using the elliptic beta integral, and the main technique in these evaluations is to make the correct choice of modular parameters to use in the elliptic beta integral.

\begin{lemma} \label{lem:int-eval1}
We have that 
\begin{multline*}
\int_\gamma \Gamma(\pm t - 2\eta; \tau, 8\eta)\theta_0(t + 4 \eta; 8\eta) e^{-2\pi i t} \theta_0(t; 2\tau)^2 \theta_0(t + \tau + 1/2; 2\tau) dt \\
=  -\frac{2}{(2\tau; 2\tau)(8\eta; 8\eta)} \Gamma(-4\eta + \tau)\Gamma(6\eta)\Gamma(-2\eta + 1/2) \Gamma(2\eta +1/2)\Gamma(-2\eta + \tau) \Gamma(6\eta + \tau)\Gamma(-2\eta + \tau + 1/2)\\\Gamma(2\eta + \tau + 1/2)\Gamma(-2\eta + 2\tau)\Gamma(8\eta + 1/2)\Gamma(12\eta +1/2)\Gamma(8\eta + \tau) \Gamma(4\eta)\Gamma(\tau+1/2).
\end{multline*}
\end{lemma}
\begin{proof}
Denote the integrand by $I(t)$.  Using the notation $\Gamma(z) := \Gamma(z; 2\tau, 8\eta)$, we notice that 
\begin{align*}
I(t) &= - \Gamma(\pm t - 2\eta; \tau, 8\eta)\theta_0(t + 4 \eta; 8\eta) \theta_0(t; 2\tau)\theta_0(t + 2\tau; 2\tau) \theta_0(t + \tau + 1/2; 2\tau)\\
&= -\Gamma(\pm t - 2\eta; \tau, 8\eta) \Gamma(\pm t + 4\eta + \tau; \tau, 8\eta) \Gamma(t + 8\eta; 2\tau, 8\eta) \Gamma(-t + 8\eta + 2\tau; 2\tau, 8\eta)\\
&\phantom{===} \Gamma(t + 2 \tau + 8\eta; 2\tau, 8\eta) \Gamma(-t + 8\eta; 2\tau, 8\eta) \Gamma(t + \tau + 8\eta + 1/2; 2\tau, 8\eta) \Gamma(-t + \tau + 8\eta + 1/2; 2\tau, 8\eta)\\
&= - \frac{\Gamma(\pm t - 2\eta) \Gamma(\pm t - 2\eta + \tau) \Gamma(\pm t + 4\eta + 2 \tau) \Gamma(\pm t + 8\eta) \Gamma(\pm t + 8\eta + 2\tau)}{\Gamma(\pm t + \tau + 1/2)}.
\end{align*}
Observe now that 
\begin{multline*}
\Gamma(\pm t + \tau + 1/2) = \frac{\Gamma(\pm t + 1/2; \tau, 8\eta)}{\Gamma(\pm t + 1/2)}= \frac{\Gamma(\pm 2t; 2\tau, 16\eta)}{\Gamma(\pm t + 1/2) \Gamma(\pm t; \tau, 8\eta)}\\ 
= \frac{\Gamma(\pm 2t)}{\Gamma(\pm t + 1/2) \Gamma(\pm 2t + 8\eta; 2\tau, 16\eta) \Gamma(\pm t)\Gamma(\pm t + \tau)}\\
= \frac{\Gamma(\pm 2t)}{\Gamma(\pm t + 1/2) \Gamma(\pm t + 4\eta)  \Gamma(\pm t + 4\eta + 1/2) \Gamma(\pm t)\Gamma(\pm t + \tau)}.
\end{multline*}
Substituting in and canceling terms, we find that 
\[
I(t) = - \frac{\Gamma(\pm t - 2\eta) \Gamma(\pm t - 2\eta + \tau) \Gamma(\pm t + 8\eta) \Gamma(\pm t + 1/2) \Gamma(\pm t + 4 \eta + 1/2) \Gamma(\pm t + \tau)}{\Gamma(\pm 2t)}.
\]
Notice now that 
\[
-2\eta - 2 \eta + \tau + 8\eta + 1/2 + 4 \eta + 1/2 + \tau = 2 \eta + 8\tau + 1,
\]
meaning that applying (an analytic continuation of) the elliptic beta integral with parameters $(-2\eta, -2\eta + \tau, 8\eta, 1/2, 4\eta +1/2, \tau)$ and periods $2\tau$ and $8\eta$ implies that
\begin{align*}
\int_\gamma I(t) dt &= -\frac{2}{(2\tau; 2\tau)(8\eta; 8\eta)} \Gamma(-4\eta + \tau)\Gamma(6\eta)\Gamma(-2\eta + 1/2) \Gamma(2\eta +1/2)\Gamma(-2\eta + \tau) \Gamma(6\eta + \tau)\Gamma(-2\eta + \tau + 1/2)\\
&\phantom{==}\Gamma(2\eta + \tau + 1/2)\Gamma(-2\eta + 2\tau)\Gamma(8\eta + 1/2)\Gamma(12\eta +1/2)\Gamma(8\eta + \tau) \Gamma(4\eta)\Gamma(\tau+1/2)\Gamma(4\eta+\tau+1/2)\\
&= -\frac{2}{(2\tau; 2\tau)(8\eta; 8\eta)} \Gamma(-4\eta + \tau)\Gamma(6\eta)\Gamma(-2\eta + 1/2) \Gamma(2\eta +1/2)\Gamma(-2\eta + \tau) \Gamma(6\eta + \tau)\Gamma(-2\eta + \tau + 1/2)\\
&\phantom{==} \Gamma(2\eta + \tau + 1/2)\Gamma(-2\eta + 2\tau)\Gamma(8\eta + 1/2)\Gamma(12\eta +1/2)\Gamma(8\eta + \tau) \Gamma(4\eta)\Gamma(\tau+1/2). \qedhere
\end{align*}
\end{proof}

\begin{lemma} \label{lem:int-eval2}
We have that 
\begin{multline*}
\int_\gamma \Gamma(\pm t - 2\eta; \tau, 8\eta)\theta_0(t + 4 \eta; 8\eta) e^{-2\pi i t} \theta_0(t + 1/2; 2\tau)^2 \theta_0(t + \tau + 1/2; 2\tau) dt \\
= \frac{2}{(2\tau; 2\tau)(8\eta; 8\eta)} \Gamma(-4\eta + \tau)\Gamma(6\eta + 1/2)\Gamma(-2\eta) \Gamma(2\eta +1/2)\Gamma(-2\eta + \tau)^2 \Gamma(-2\eta + 2\tau)\\
\Gamma(8\eta + 1/2)\Gamma(12\eta)\Gamma(8\eta + \tau + 1/2) \Gamma(4\eta + 1/2)\Gamma(\tau). 
\end{multline*}
\end{lemma}
\begin{proof}
Denote the integrand by $I(t)$.  Using the notation $\Gamma(z) := \Gamma(z; 2\tau, 8\eta)$, we notice that 
\begin{align*}
I(t) &= \Gamma(\pm t - 2\eta; \tau, 8\eta)\theta_0(t + 4 \eta; 8\eta) \theta_0(t + 1/2; 2\tau)\theta_0(t + 2\tau + 1/2; 2\tau) \theta_0(t + \tau + 1/2; 2\tau)\\
&= \Gamma(\pm t - 2\eta; \tau, 8\eta) \Gamma(\pm t + 4\eta + \tau; \tau, 8\eta) \Gamma(t + 8\eta + 1/2; 2\tau, 8\eta) \Gamma(-t + 8\eta + 2\tau + 1/2; 2\tau, 8\eta)\\
&\phantom{===} \Gamma(t + 2 \tau + 8\eta + 1/2; 2\tau, 8\eta) \Gamma(-t + 8\eta + 1/2; 2\tau, 8\eta) \Gamma(t + \tau + 8\eta + 1/2; 2\tau, 8\eta) \Gamma(-t + \tau + 8\eta + 1/2; 2\tau, 8\eta)\\
&= \frac{\Gamma(\pm t - 2\eta) \Gamma(\pm t - 2\eta + \tau) \Gamma(\pm t + 4\eta + 2 \tau) \Gamma(\pm t + 8\eta + 1/2) \Gamma(\pm t + 8\eta + 2\tau + 1/2)}{\Gamma(\pm t + \tau + 1/2)}.
\end{align*}
Observe now that 
\begin{multline*}
\Gamma(\pm t + \tau + 1/2) = \frac{\Gamma(\pm t + 1/2; \tau, 8\eta)}{\Gamma(\pm t + 1/2)}
= \frac{\Gamma(\pm 2t; 2\tau, 16\eta)}{\Gamma(\pm t + 1/2) \Gamma(\pm t; \tau, 8\eta)}\\ 
= \frac{\Gamma(\pm 2t)}{\Gamma(\pm t + 1/2) \Gamma(\pm 2t + 8\eta; 2\tau, 16\eta) \Gamma(\pm t)\Gamma(\pm t + \tau)}\\ 
= \frac{\Gamma(\pm 2t)}{\Gamma(\pm t + 1/2) \Gamma(\pm t + 4\eta)  \Gamma(\pm t + 4\eta + 1/2) \Gamma(\pm t)\Gamma(\pm t + \tau)}.
\end{multline*}
Substituting in and canceling terms, we find that 
\[
I(t) = - \frac{\Gamma(\pm t - 2\eta) \Gamma(\pm t - 2\eta + \tau) \Gamma(\pm t + 8\eta + 1/2) \Gamma(\pm t) \Gamma(\pm t + 4 \eta + 1/2) \Gamma(\pm t + \tau)}{\Gamma(\pm 2t)}.
\]
Notice now that 
\[
-2\eta - 2 \eta + \tau + 8\eta + 1/2 + 4 \eta + 1/2 + \tau = 2 \eta + 8\tau + 1,
\]
meaning that applying (an analytic continuation of) the elliptic beta integral with parameters $(-2\eta, -2\eta + \tau, 8\eta + 1/2, 0, 4\eta +1/2, \tau)$ and periods $2\tau$ and $8\eta$ implies that
\begin{align*}
\int_\gamma I(t) dt &= \frac{2}{(2\tau; 2\tau)(8\eta; 8\eta)} \Gamma(-4\eta + \tau)\Gamma(6\eta + 1/2)\Gamma(-2\eta) \Gamma(2\eta +1/2)\Gamma(-2\eta + \tau)\Gamma(6\eta + \tau + 1/2)\Gamma(-2\eta + \tau)\\
&\phantom{==} \Gamma(2\eta + \tau + 1/2)\Gamma(-2\eta + 2\tau)\Gamma(8\eta + 1/2)\Gamma(12\eta)\Gamma(8\eta + \tau + 1/2) \Gamma(4\eta + 1/2)\Gamma(\tau)\Gamma(4\eta+\tau+1/2)\\
&= \frac{2}{(2\tau; 2\tau)(8\eta; 8\eta)} \Gamma(-4\eta + \tau)\Gamma(6\eta + 1/2)\Gamma(-2\eta) \Gamma(2\eta +1/2)\Gamma(-2\eta + \tau)^2 \Gamma(-2\eta + 2\tau)\\
&\phantom{==} \Gamma(8\eta + 1/2)\Gamma(12\eta)\Gamma(8\eta + \tau + 1/2) \Gamma(4\eta + 1/2)\Gamma(\tau). \qedhere
\end{align*}
\end{proof}

\section{Felder-Varchenko functions and elliptic Macdonald polynomials} \label{sec:fv}

In this section, we introduce the Felder-Varchenko functions specialized to the three-dimensional irreducible representation of $U_q(\sl_2)$, explain their relation to the elliptic Macdonald polynomials defined in \cite{FV4}, and prove Theorems \ref{thm:fv-val1}, \ref{thm:fv-val2}, and \ref{thm:ellmac-val} on three special values for these functions.  We provide also some motivation for the form of these special values coming from $SL(3, \ZZ)$ modular properties of the Felder-Varchenko functions.  In this section, we will only use additive notation.

\subsection{Felder-Varchenko functions}

For modular parameters $\tau, \sigma$ with $\Imm(\tau), \Imm(\sigma) > 0$, the Felder-Varchenko function corresponding to the three-dimensional representation of $U_q(\sl_2)$ and studied in \cite{FTV, FTV2, FV2, FV} is the theta hypergeometric integral
\[
u(\lambda, \mu, \tau, \sigma, \eta) := e^{-\frac{\pi i \lambda \mu}{2\eta}} \int_\gamma \Omega_{2\eta}(t; \tau, \sigma) \frac{\theta(t + \lambda; \tau)}{\theta(t - 2\eta; \tau)} \frac{\theta(t + \mu; \sigma)}{\theta(t - 2\eta; \sigma)} dt,
\]
where the phase function is defined by
\[
\Omega_{2\eta}(t; \tau, \sigma) = \frac{\Gamma(t + 2\eta; \tau, \sigma)}{\Gamma(t - 2\eta; \tau, \sigma)}
\]
and the cycle $\gamma$ is the interval $[-1/2, 1/2]$ for $\Imm(\eta) > 0$ and is deformed to have the same sets of poles above and below for other $\eta$. These functions were initially defined as hypergeometric integral solutions to the $q$-KZB and $q$-KZB heat equations.

\subsection{Theta functions of level $\kappa$}

For a modular parameter $\tau$ with $\Imm(\tau) > 0$, a holomorphic function $f$ is a theta function of level $\kappa \geq 0$ if 
\[
f(\lambda + 2 r + 2s\tau) = e^{-2\pi i \kappa(s^2 \tau + s\lambda)} f(\lambda)
\]
for integers $r, s$.  A particular theta function of level $\kappa$ is given by
\[
\theta_{\mu, \kappa}(\lambda; \tau) := \sum_{n \in \ZZ + \frac{\mu}{2\kappa}} e^{2\pi i \kappa (n^2\tau + n \lambda)}.
\]

\begin{lemma} \label{lem:theta-prod}
We have
\[
\theta_{\mu, \kappa}(\lambda; \tau) = e^{\pi i \tau \frac{\mu^2}{2\kappa} + \pi i \lambda \mu} (2\kappa\tau; 2\kappa\tau) \theta_0(1/2 + \mu \tau + \kappa \tau + \kappa \lambda; 2\kappa\tau).
\]
\end{lemma}
\begin{proof}
By the Jacobi triple product formula, we have
\begin{align*}
\theta_{\mu, \kappa}(\lambda; \tau) &= \sum_{n \in \ZZ} e^{2\pi i \kappa \tau n^2 + 2\pi i \tau \mu n + \pi i \tau \frac{\mu^2}{2\kappa} + 2 \pi i \kappa \lambda n + \pi i \lambda \mu}\\
&= e^{\pi i \tau \frac{\mu^2}{2\kappa} + \pi i \lambda \mu} \sum_{n \in \ZZ} e^{2\pi i \kappa \tau n^2 + (2\pi i \tau \mu + 2 \pi i \kappa \lambda) n}\\
&= e^{\pi i \tau \frac{\mu^2}{2\kappa} + \pi i \lambda \mu} (2\kappa\tau; 2\kappa\tau)(1/2 + \mu \tau + \kappa \tau + \kappa \lambda; 2\kappa\tau)(1/2 + \kappa \tau - \mu \tau - \kappa \lambda; 2\kappa\tau)\\
&= e^{\pi i \tau \frac{\mu^2}{2\kappa} + \pi i \lambda \mu} (2\kappa\tau; 2\kappa\tau) \theta_0(1/2 + \mu \tau + \kappa \tau + \kappa \lambda; 2\kappa\tau). \qedhere
\end{align*}
\end{proof}

\subsection{Hypergeometric theta functions and elliptic Macdonald polynomials}

In \cite{FV4}, Felder-Varchenko used theta functions of their hypergeometric integrals to define elliptic versions of Macdonald polynomials corresponding to type $\widehat{A}_1$ and $t = q^2$.  For a positive integer level $\kappa \geq 4$ and $\mu \neq \pm 1 \pmod{\kappa}$, they defined the non-symmetric hypergeometric theta function of level $\kappa + 2$ by 
\[
\wDelta_{\mu, \kappa}(\lambda; \tau, \eta) := \sum_{j \in 2\kappa \ZZ + \mu} u(\lambda, 2 \eta j, \tau, -2\eta \kappa, \eta) Q(2\eta j, -2\eta \kappa, \eta) e^{\pi i \frac{\tau + 4\eta}{2\kappa} j^2},
\]
where
\[
Q(\mu; \sigma, \eta) := \frac{\theta(4\eta; \sigma) \theta'(0; \sigma)}{\theta(\mu - 2\eta; \sigma) \theta(\mu + 2\eta; \sigma)}.
\]
They defined further the symmetrized hypergeometric theta function by 
\[
\Delta_{\mu, \kappa}(\lambda; \tau, \eta) := \wDelta_{\mu, \kappa}(\lambda; \tau, \eta) - \wDelta_{\mu, \kappa}(-\lambda; \tau, \eta).
\]
These functions admit the following convergence properties and integral form.

\begin{prop}[{\cite[Theorem 3.1]{FV4}}] \label{prop:htf-prop}
Suppose that $\Imm(\eta) < 0$, $\Imm(\tau) > 0$, and $j \tau + 4 \eta \notin \ZZ$ for any positive integer $j$.  Then $\wDelta_{\mu, \kappa}(\lambda; \tau, \eta)$ converges to a holomorphic function of $\lambda$ and admits the integral expression
\[
\wDelta_{\mu, \kappa}(\lambda; \tau, \eta) = e^{\frac{2\pi i \eta}{\kappa} \mu^2} Q(2\eta\mu; -2\eta\kappa, \eta) \wI_{\mu, \kappa}(\lambda; \tau, \eta)
\]
for 
\begin{multline*}
\wI_{\mu, \kappa}(\lambda; \tau, \eta) := \int_\gamma \Omega_{2\eta}(t; \tau, -2\eta\kappa) \frac{\theta(t + \lambda; \tau)}{\theta(t - 2\eta; \tau)} \frac{\theta(t + 2\eta \mu; -2\eta \kappa)}{\theta(t - 2\eta; -2\eta\kappa)} e^{-2\pi i \mu t/\kappa} \theta_{\mu, \kappa}(\frac{2}{\kappa}t - \lambda; \tau) dt\\
= e^{\pi i \tau \frac{\mu^2}{2\kappa} - \pi i \lambda \mu} (2 \kappa \tau; 2\kappa \tau) \int_\gamma \Omega_{2\eta}(t; \tau, -2\eta\kappa) \frac{\theta(t + \lambda; \tau)}{\theta(t - 2\eta; \tau)} \frac{\theta(t + 2\eta \mu; -2\eta \kappa)}{\theta(t - 2\eta; -2\eta\kappa)} \theta_0(1/2 + \mu \tau + \kappa \tau - \kappa \lambda + 2t; 2\kappa\tau) dt.
\end{multline*}
\end{prop}

The elliptic Macdonald polynomial for $t = q^2$ was defined in \cite[Section 5.2]{FV4} by 
\[
P_{\mu, \kappa}(\lambda; \tau, \eta) := e^{-\pi i \frac{4\eta + \tau}{2\kappa}(\mu + 2)^2 + \pi i 3\tau/4} \frac{\Delta_{\mu + 2, \kappa}(\lambda; \tau, \eta)}{\theta(\lambda - 2\eta; \tau)\theta(\lambda; \tau)\theta(\lambda + 2\eta; \tau)}.
\]
In \cite[Theorem 5.2]{FV4}, it was shown that as $\tau \to i \infty$, $P_{\mu, \kappa}(\lambda; \tau, \eta)$ converges to a constant multiple of the ordinary Macdonald polynomial $P_\mu(e^{\pi i \lambda}; q, q^2)$ for $\sl_2$.  It was conjectured in \cite{FV4} and proven in \cite{Sun:qafftr} that these elliptic Macdonald polynomials are related to the affine Macdonald polynomials defined in \cite{EK3} by a simple renormalization; we will discuss and exploit this relation further in Section \ref{sec:aff-ell}.

\subsection{Special values of Felder-Varchenko functions and elliptic Macdonald polynomials}

We are now ready to state our main results, which concern special values of parameters for which the integral formulas for Felder-Varchenko functions and elliptic Macdonald polynomials admit explicit evaluation.  These results correspond to the integral evaluations in Section \ref{sec:int-eval}.

\begin{remark}
In Theorems \ref{thm:fv-val1} and \ref{thm:fv-val2}, the cycle $\gamma$ in the definition of the Felder-Varchenko function must be deformed to consider $\eta = \pm 1/8$.  In this case, the deformed cycle lies above the pole at $t = -2\eta$ and below the pole at $t = 2\eta$ and separates all other poles above and below the real line.
\end{remark}

\begin{thm} \label{thm:fv-val1}
We have that 
\[
u(1/2, 1/2, \tau, \sigma, -1/8) = - (1 + i) \frac{\Gamma(3/4; \tau, \sigma)}{\Gamma(1/4; \tau, \sigma)} \frac{1}{(\tau; \tau)(\tau+1/2; 2\tau)} \frac{1}{(\sigma; \sigma)(\sigma + 1/2; 2\sigma)}.
\]
\end{thm}
\begin{proof}
By definition and Theorem \ref{thm:eval2}, we find that 
\begin{align*}
u(1/2, 1/2, \tau, \sigma, -1/8) &= - \int_\gamma \frac{\Gamma(t - 1/4; \tau, \sigma)}{\Gamma(t + 1/4; \tau, \sigma)} \frac{\theta(t + 1/2; \tau)}{\theta(t + 1/4; \tau)} \frac{\theta(t + 1/2; \sigma)}{\theta(t + 1/4; \sigma)} dt\\
&= - e^{-\pi i/2} \int_\gamma \frac{\Gamma(t - 1/4; \tau, \sigma)}{\Gamma(t + 1/4; \tau, \sigma)} \frac{\theta_0(t + 1/2; \tau)}{\theta_0(t + 1/4; \tau)} \frac{\theta_0(t + 1/2; \sigma)}{\theta_0(t + 1/4; \sigma)} dt\\
&= - (1 + i) \frac{\Gamma(3/4; \tau, \sigma)}{\Gamma(1/4; \tau, \sigma)} \frac{1}{(\tau; \tau)(\tau+1/2; 2\tau)} \frac{1}{(\sigma; \sigma)(\sigma + 1/2; 2\sigma)}. \qedhere
\end{align*}
\end{proof}

\begin{thm} \label{thm:fv-val2}
We have that 
\[
u(1/2, 1/2, \tau, \sigma, 1/8) = - (1 - i) \frac{\Gamma(3/4; \tau, \sigma)}{\Gamma(1/4; \tau, \sigma)} \frac{1}{(\tau; \tau)(\tau+1/2; 2\tau)} \frac{1}{(\sigma; \sigma)(\sigma + 1/2; 2\sigma)}.
\]
\end{thm}
\begin{proof}
By definition and Theorem \ref{thm:eval1}, we find that 
\begin{align*}
u(1/2, 1/2, \tau, \sigma, 1/8) &= - \int_\gamma \frac{\Gamma(t + 1/4; \tau, \sigma)}{\Gamma(t - 1/4; \tau, \sigma)} \frac{\theta(t + 1/2; \tau)}{\theta(t - 1/4; \tau)} \frac{\theta(t + 1/2; \sigma)}{\theta(t - 1/4; \sigma)} dt\\
&= - e^{-3\pi i/2} \int_\gamma \frac{\Gamma(t + 1/4; \tau, \sigma)}{\Gamma(t - 1/4; \tau, \sigma)} \frac{\theta_0(t + 1/2; \tau)}{\theta_0(t - 1/4; \tau)} \frac{\theta_0(t + 1/2; \sigma)}{\theta_0(t - 1/4; \sigma)} dt\\
&= - (1 - i) \frac{\Gamma(3/4; \tau, \sigma)}{\Gamma(1/4; \tau, \sigma)} \frac{1}{(\tau; \tau)(\tau+1/2; 2\tau)} \frac{1}{(\sigma; \sigma)(\sigma + 1/2; 2\sigma)}. \qedhere
\end{align*}
\end{proof}

\begin{thm} \label{thm:ellmac-val}
We have that 
\[
P_{0, 4}(\lambda; \tau, \eta) = - 2\pi \frac{\Gamma(-6\eta; \tau, -8\eta)}{\Gamma(-2\eta; \tau, -8\eta)} \frac{1}{\theta_0(4\eta; \tau)(\tau;\tau)^3} \frac{(-4\eta; -4\eta)}{(-2\eta; -4\eta)}.
\]
\end{thm}
\begin{proof}
By Proposition \ref{prop:htf-prop}, we have that 
\[
\wDelta_{2, 4}(\lambda; \tau, \eta) = e^{2\pi i \eta} Q(4\eta; -8\eta, \eta) \wI_{2, 4}(\lambda; \tau, \eta)
\]
and that
\begin{align*}
\wI_{2, 4}(\lambda; \tau, \eta) &= e^{\pi i \tau/2 - 3 \pi i \lambda - 8 \pi i \eta} (8\tau; 8\tau)\\
&\phantom{====} \int_\gamma \Omega_{2\eta}(t; \tau, -8\eta) \frac{\theta_0(t + \lambda; \tau)}{\theta_0(t - 2\eta; \tau)} \frac{\theta_0(t + 4\eta; -8\eta)}{\theta_0(t - 2\eta; -8\eta)} \theta_0(1/2 + 6\tau - 4\lambda + 2t; 8\tau) dt\\
&= e^{\pi i \tau/2 - 8 \pi i \eta} (8\tau; 8\tau) \wI(\lambda; \tau, -\eta)
\end{align*}
where $\wI(\lambda; \tau, \eta)$ is defined in (\ref{eq:asym-int}).  We conclude by Theorem \ref{thm:eval3} that
\begin{align*}
\Delta_{2, 4}(\lambda; \tau, \eta) &= e^{\pi i \tau/2 - 6\pi i \eta} Q(4\eta; -8\eta, \eta) (8\tau; 8\tau) I(\lambda; \tau, -\eta)\\
&= e^{\pi i \tau/2 + 6\pi i \eta} Q(4\eta; -8\eta, \eta)  \frac{\Gamma(-6\eta; \tau, -8\eta)}{\Gamma(-2\eta; \tau, -8\eta)} \frac{e^{-3\pi i \lambda} \theta_0(\lambda; \tau)\theta_0(\lambda - 2\eta; \tau)\theta_0(\lambda + 2\eta; \tau)}{\theta_0(4\eta; \tau)(-4\eta; -4\eta)(-2\eta+1/2; -2\eta)},
\end{align*}
where we note that
\[
Q(4\eta; -8\eta, \eta) = \frac{\theta(4\eta; -8\eta) \theta'(0, -8\eta)}{\theta(2\eta; -8\eta) \theta(6\eta; -8\eta)}= -2\pi i e^{4\pi i \eta}  \frac{\theta_0(4\eta; -8\eta) (-8\eta; -8\eta)^2}{\theta_0(2\eta; -8\eta) \theta_0(6\eta; -8\eta)}.
\]
Therefore, we find that 
\begin{align*}
P_{0, 4}(\lambda; \tau, \eta) &= e^{-2\pi i \eta + \pi i \tau/4} \frac{\Delta_{2, 4}(\lambda; \tau, \eta)}{\theta(\lambda - 2\eta; \tau)\theta(\lambda; \tau)\theta(\lambda + 2\eta; \tau)}\\
&= 2\pi e^{8\pi i \eta} \frac{\Gamma(-6\eta; \tau, -8\eta)}{\Gamma(-2\eta; \tau, -8\eta)} \frac{1}{(\tau; \tau)^3 \theta_0(4\eta; \tau)} \frac{\theta_0(4\eta;-8\eta) (-8\eta; -8\eta)^2}{\theta_0(2\eta; -8\eta)\theta_0(6\eta; -8\eta) (-4\eta; -4\eta)(-2\eta + 1/2; -2\eta)}.
\end{align*}
Notice now that 
\begin{align*}
&\frac{\theta_0(4\eta;-8\eta) (-8\eta; -8\eta)^2}{\theta_0(2\eta; -8\eta)\theta_0(6\eta; -8\eta) (-4\eta; -4\eta)(-2\eta + 1/2; -2\eta)}\\
&\phantom{======}= \frac{(4\eta; -8\eta)(-12\eta; -8\eta) (-8\eta; -8\eta)^2}{(2\eta; -8\eta)(-10\eta; -8\eta)(6\eta; -8\eta)(-14\eta; -8\eta) (-4\eta;-4\eta)(-2\eta+1/2;-2\eta)}\\
&\phantom{======}= - \frac{e^{8\pi i \eta}(-4\eta; -4\eta)^2}{(-4\eta; -4\eta)(-2\eta+1/2;-2\eta)} \frac{1}{e^{4\pi i \eta + 12\pi i \eta} (-2\eta; -4\eta)^2}\\
&\phantom{======}= - e^{-8\pi i \eta} \frac{(-4\eta; -4\eta)}{(-2\eta; -4\eta)},
\end{align*}
from which we conclude that 
\[
P_{0, 4}(\lambda; \tau, \eta) = - 2\pi \frac{\Gamma(-6\eta; \tau, -8\eta)}{\Gamma(-2\eta; \tau, -8\eta)} \frac{1}{\theta_0(4\eta; \tau)(\tau;\tau)^3} \frac{(-4\eta; -4\eta)}{(-2\eta; -4\eta)}. \qedhere
\]
\end{proof}

\begin{thm} \label{thm:ellmac-eval}
We have that 
\[
P_{\mu, \kappa}(4\eta; -8\eta, \eta) = -2 \pi e^{-12\pi i \eta - 2\pi i (\mu + 2)\eta} \frac{\Gamma(-6\eta; -2\kappa\eta, -8\eta)}{\Gamma(-2\eta; -2\kappa\eta, -8\eta)} \frac{\theta_0(2(\mu + 2)\eta; -2\kappa\eta) (-2\kappa \eta; -2\kappa\eta)^2}{(-8\eta; -8\eta) (-4\eta; -4\eta)^2 (-2\eta; -2\eta)}.
\]
\end{thm}
\begin{proof}
First, notice that
\begin{align*}
\wI_{\mu, \kappa}(4\eta; -8\eta, \eta) &= e^{-\frac{4\pi i \eta}{\kappa} \mu^2 - 4\pi i \eta \mu} (-16 \kappa \eta; -16 \kappa \eta)\\
& \int_\gamma \Omega_{2\eta}(t; -8\eta, -2\kappa \eta) \frac{\theta(t + 4\eta; -8\eta)}{\theta(t - 2\eta; -8\eta)} \frac{\theta(t + 2 \mu \eta; -2 \kappa \eta)}{\theta(t - 2\eta; -2\kappa\eta)} \theta_0(1/2 - 8\mu\eta - 12\kappa\eta + 2t; -16\kappa\eta) dt\\
&= e^{-\frac{4\pi i \eta}{\kappa} \mu^2 - 8\pi i \eta} (-16 \kappa \eta; -16 \kappa \eta) \wI(2\mu \eta; -2\kappa\eta, -\eta).
\end{align*}
On the other hand, we see that 
\begin{align*}
\wI_{\mu, \kappa}(-4\eta; -8\eta, \eta) &= e^{-\frac{4\pi i \eta}{\kappa} \mu^2 + 4\pi i \eta \mu} (-16 \kappa \eta; -16 \kappa \eta)\\
& \int_\gamma \Omega_{2\eta}(t; -8\eta, -2\kappa \eta) \frac{\theta(t - 4\eta; -8\eta)}{\theta(t - 2\eta; -8\eta)} \frac{\theta(t + 2 \mu \eta; -2 \kappa \eta)}{\theta(t - 2\eta; -2\kappa\eta)} \theta_0(1/2 - 8\mu\eta - 4\kappa\eta + 2t; -16\kappa\eta) dt\\
&= e^{-\frac{4\pi i \eta}{\kappa} \mu^2 + 2\pi i \eta\mu} (-16 \kappa \eta; -16 \kappa \eta)\\
& \int_\gamma \Omega_{2\eta}(t; -8\eta, -2\kappa \eta) \frac{\theta_0(t - 4\eta; -8\eta)}{\theta_0(t - 2\eta; -8\eta)} \frac{\theta_0(t + 2 \mu \eta; -2 \kappa \eta)}{\theta_0(t - 2\eta; -2\kappa\eta)} \theta_0(1/2 - 8\mu\eta - 4\kappa\eta + 2t; -16\kappa\eta) dt.
\end{align*}
Denote the integrand by 
\[
f(t) = e^{2\pi i \eta \mu} \Omega_{2\eta}(t; -8\eta, -2\kappa \eta) \frac{\theta_0(t - 4\eta; -8\eta)}{\theta_0(t - 2\eta; -8\eta)} \frac{\theta_0(t + 2 \mu \eta; -2 \kappa \eta)}{\theta_0(t - 2\eta; -2\kappa\eta)} \theta_0(1/2 - 8\mu\eta - 4\kappa\eta + 2t; -16\kappa\eta).
\]
Notice now that 
\begin{align*}
\frac{\Omega_{2\eta}(t; -8\eta, -2\kappa \eta)}{\theta_0(t - 2\eta; -8\eta) \theta_0(t - 2\eta; -2\kappa\eta)} &= \frac{\Gamma(t + 2 \eta; -8\eta, -2\kappa \eta)}{\Gamma(t - 2\eta; -8\eta, -2\kappa\eta)} \frac{1}{\theta_0(t - 2\eta; -8\eta) \theta_0(t - 2\eta; -2\kappa\eta)}\\
&= \frac{\Gamma(t + 2\eta; -8\eta, -2\kappa \eta) \Gamma(-t - 6\eta - 2\kappa\eta; -8\eta, -2\kappa\eta)}{\theta_0(t - 2\eta; -8\eta) \theta_0(t - 2\eta; -2\kappa\eta)}\\
&= - e^{-2\pi i t + 4\pi i \eta} \Gamma(\pm t + 2\eta; -8\eta, -2\kappa \eta).
\end{align*}
We conclude that 
\begin{align*}
f(-t) &= - e^{2\pi i t + 6\pi i \eta} \Gamma(\pm t + 2\eta; -8\eta, -2\kappa \eta) \theta_0(t - 4\eta; -8\eta) \theta_0(t - 2\mu \eta - 2\kappa \eta; -2\kappa\eta) \theta_0(1/2 + 8\mu \eta - 12\kappa \eta + 2t; -16\kappa\eta)\\
&= e^{-8\pi i \eta + 6\pi i \mu \eta} \frac{\Gamma(t + 2\eta; -8\eta, -2\kappa\eta)}{\Gamma(t - 2\eta; -8\eta, -2\kappa\eta)}\frac{\theta_0(t + 4\eta; -8\eta)}{\theta_0(t - 2\eta; -8\eta)} \frac{\theta_0(t - 2\mu \eta; -2\kappa\eta)}{\theta_0(t - 2\eta; -2\kappa\eta)} \theta_0(1/2 + 8\mu \eta - 12\kappa \eta + 2t; -16\kappa\eta).
\end{align*}
Because the cycle $\gamma$ is invariant under $t \mapsto -t$, we conclude that
\begin{multline*}
\wI_{\mu, \kappa}(-4\eta; -8\eta, \eta) = e^{-\frac{4\pi i \eta}{\kappa}\mu^2} (-16\kappa\eta; -16\kappa\eta) \int_\gamma f(t) dt\\ = e^{-\frac{4\pi i \eta}{\kappa}\mu^2} (-16\kappa\eta; -16\kappa\eta) \int_\gamma f(-t) dt = e^{-\frac{4\pi i \eta}{\kappa}\mu^2 - 8\pi i \eta} (-16\kappa\eta; -16\kappa\eta) \wI(-2\mu\eta; -2\kappa\eta, -\eta).
\end{multline*}
By Proposition \ref{prop:htf-prop}, we have that 
\begin{align*}
\Delta_{\mu, \kappa}(4\eta; -8\eta, \eta) &= e^{\frac{2\pi i \eta}{\kappa} \mu^2} Q(2\eta \mu; -2\kappa\eta, \eta) \Big(\wI_{\mu, \kappa}(4\eta; -8\eta, -\eta) - \wI_{\mu, \kappa}(-4\eta; -8\eta, -\eta)\Big)\\
&= e^{-\frac{2\pi i \eta}{\kappa}\mu^2 - 8\pi i \eta} (-16\kappa \eta; -16\kappa \eta) Q(2\eta \mu; -2\kappa\eta, \eta) I(2\mu\eta; -2\kappa\eta, -\eta).
\end{align*}
Applying Theorem \ref{thm:eval3} and noting that
\begin{align*}
Q(2\mu \eta; -2\kappa\eta, \eta) &= \frac{\theta(4\eta; -2\kappa \eta)\theta'(0, -2\kappa\eta)}{\theta(2\mu \eta - 2\mu; -2\kappa\eta) \theta(2\mu \eta + 2\eta; -2\kappa\eta)}\\
&= -2\pi i e^{-4\pi i \eta + 4\pi i \mu \eta} \frac{\theta_0(4\eta; -2\kappa\eta) (-2\kappa \eta, -2\kappa\eta)^2}{\theta_0(2\mu\eta - 2\eta; -2\kappa\eta)\theta_0(2\mu \eta + 2\eta; -2\kappa\eta)},
\end{align*}
we conclude that
\begin{align*}
\Delta_{\mu, \kappa}(4\eta; -8\eta, \eta) &= - 2\pi i e^{-\frac{2\pi i \eta}{\kappa}\mu^2 - 2\pi i \mu\eta} \frac{\Gamma(-6\eta; -2\kappa\eta, -8\eta)}{\Gamma(-2\eta; -2\kappa\eta, -8\eta)} \frac{\theta_0(2\mu\eta; -2\kappa\eta) (-2\kappa \eta; -2\kappa\eta)^2}{(-4\eta; -4\eta)(-2\eta+1/2; -2\eta)}.
\end{align*}
This finally implies that
\begin{align*}
P_{\mu, \kappa}(4\eta; -8\eta, \eta) &= e^{\frac{2\pi i \eta}{\kappa}(\mu + 2)^2 - 6\pi i \eta} \frac{\Delta_{\mu + 2, \kappa}(4\eta; -8\eta, \eta)}{\theta(2\eta; -8\eta)\theta(4\eta; -8\eta)\theta(6\eta; -8\eta)}\\
&= 2 \pi e^{12\pi i \eta - 2\pi i (\mu + 2)\eta} \frac{\Gamma(-6\eta; -2\kappa\eta, -8\eta)}{\Gamma(-2\eta; -2\kappa\eta, -8\eta)} \theta_0(2(\mu + 2)\eta; -2\kappa\eta) (-2\kappa \eta; -2\kappa\eta)^2\\
&\phantom{==} \frac{1}{(-8\eta; -8\eta)^3 \theta_0(2\eta; -8\eta)\theta_0(4\eta; -8\eta) \theta_0(6\eta; -8\eta)(-4\eta; -4\eta)(-2\eta+1/2; -2\eta)}.
\end{align*}
Notice now that 
\begin{align*}
(-8\eta; -8\eta)^3& \theta_0(2\eta; -8\eta)\theta_0(4\eta; -8\eta) \theta_0(6\eta; -8\eta)(-4\eta; -4\eta)(-2\eta+1/2; -2\eta)\\
&= (-8\eta; -8\eta)^3 (4\eta; -8\eta)(-12\eta; -8\eta) \theta_0(6\eta; -4\eta) (-4\eta; -4\eta)(-2\eta+1/2; -2\eta)\\
&= e^{16\pi i \eta} (1 - e^{8\pi i \eta}) (-8\eta; -8\eta) (-8\eta; -4\eta) (-4\eta; -4\eta)^2 (-2\eta; -4\eta)^2 (-2\eta+1/2; -2\eta)\\
&= - e^{24\pi i \eta} (-8\eta; -8\eta) (-4\eta; -4\eta)^2 (-2\eta; -2\eta).
\end{align*}
We conclude that 
\[
P_{\mu, \kappa}(4\eta; -8\eta, \eta) = -2 \pi e^{-12\pi i \eta - 2\pi i (\mu + 2)\eta} \frac{\Gamma(-6\eta; -2\kappa\eta, -8\eta)}{\Gamma(-2\eta; -2\kappa\eta, -8\eta)} \frac{\theta_0(2(\mu + 2)\eta; -2\kappa\eta) (-2\kappa \eta; -2\kappa\eta)^2}{(-8\eta; -8\eta) (-4\eta; -4\eta)^2 (-2\eta; -2\eta)}.\qedhere
\]
\end{proof}

\subsection{$SL(3, \ZZ)$ modular properties}

In this section we provide some motivation for Theorems \ref{thm:fv-val1}, \ref{thm:fv-val2}, and \ref{thm:ellmac-val} via the $SL(3, \ZZ)$ modular properties of Felder-Varchenko functions and the elliptic gamma function.  We were able to use this philosophy to conjecture and numerically verify their statements before finding their proofs via the elliptic hypergeometric integrals in Section \ref{sec:int-eval}.  In \cite[Theorem 4.1]{FV4}, modular transformation properties are given for hypergeometric theta functions.  When translated into modular properties for the elliptic Macdonald polynomials and specialized to $P_{0, 4}$, they become the following three term relations.
\begin{prop} \label{prop:ellmac-mod}
The function $P_{0, 4}(\lambda; \tau, \eta)$ satisfies the modular properties
\begin{equation} \label{eq:mod-prop1}
P_{0, 4}(\lambda; \tau, \eta) S^-(\tau, \eta) P_{0, 4}(\lambda; - 1/\tau, \eta/\tau)^{-1} = 4 \sqrt{2} \pi i\tau \exp\Big(\pi i \frac{4 + 216\eta^2 - 42 \eta(\tau - 1) + 3\tau + 4\tau^2}{12\tau}\Big).
\end{equation}
and
\begin{equation} \label{eq:mod-prop2}
P_{0, 4}(\lambda; \tau, \eta) S^+(\tau, \eta) P_{0, 4}(\lambda; -1/\tau, -\eta/\tau)^{-1} = -4 \sqrt{2}\pi i \tau \exp\Big(\pi i \frac{4 + 216 \eta^2 - 42 \eta(1 + \tau) - 3 \tau + 4\tau^2}{12\tau}\Big)
\end{equation}
for the quantities
\begin{align*}
S^-(\tau, \eta) &:= -2 \frac{\theta(1/2; \tau/8\eta)\theta'(0; \tau/8\eta)}{\theta(3/4; \tau/8\eta) \theta(1/4; \tau/8\eta)} u(1/2, 1/2, 1/8\eta, \tau/8\eta, -1/8)\\
S^+(\tau, \eta) &:= 2 \frac{\theta(1/2; -\tau/8\eta)\theta'(0; -\tau/8\eta)}{\theta(1/4; -\tau/8\eta) \theta(3/4; -\tau/8\eta)} u(1/2, -1/2, 1/8\eta, -\tau/8\eta, 1/8).
\end{align*}
\end{prop}
We guessed the statements of Theorems \ref{thm:fv-val1}, \ref{thm:fv-val2}, and \ref{thm:ellmac-val} by creating candidate expressions which are products of elliptic gamma functions and theta functions with the correct modular parameters, degenerate to known trigonometric and classical limits of these expressions, and satisfy the modular properties (\ref{eq:mod-prop1}) and (\ref{eq:mod-prop2}) as a result of the modular transformations (\ref{eq:theta-mod}) and (\ref{eq:ellgam-mod}).  Before finding proofs for Theorems \ref{thm:fv-val1}, \ref{thm:fv-val2}, and \ref{thm:ellmac-val}, we were then able to verify the conjectured expressions for generic special values of $\tau$ and $\eta$ by numerical integration.

\section{Affine Macdonald conjectures} \label{sec:aff-mac}

In this section, we relate our results on special values of elliptic Macdonald polynomials to special cases of the constant term and evaluation conjectures for the affine Macdonald polynomials defined by Etingof-Kirillov~Jr. in \cite{EK3}.  We first introduce the affine Macdonald polynomials, relate them to elliptic Macdonald polynomials using the results of \cite{Sun:qafftr}, and then state and prove the relevant conjectures.  In this section, we will only use multiplicative notation.

\subsection{The quantum affine algebra $U_q(\asl_{n})$} 

Let $\alpha_i$, $i = 1, \ldots, n - 1$ be the simple roots for $\sl_{n}$, $\theta$ the highest root, $\rho = \frac{1}{2} \sum_{\alpha > 0} \alpha$, and note that the dual Coxeter number is $\ch = 1 + (\theta, \rho) = n$. Let $A = (a_{ij})_{i, j = 0}^{n-1}$ be the extended Cartan matrix of $\asl_n$.  Let the Cartan and dual Cartan algebras be 
\[
\whh = \hh \oplus \CC c \oplus \CC d \text{ and } \whh^* = \hh^* \oplus \CC \Lambda_0 \oplus \CC \delta,
\]
with $\Lambda_0 = c^*$ and $\delta = d^*$. Take $\alpha_0 := \delta - \theta \in \whh^*$.  The algebra $\wsl_n$ admits a non-degenerate invariant form $(-,-)$ whose restriction to $\whh$ satisfies
\[
(d, d) = 0 \qquad (c, d) = 1 \qquad (\alpha_i, \alpha_i) = 2 \text{ for $i > 0$}
\]
and agrees with the standard non-degenerate form on $\hh$. Fix an orthonormal basis $\{x_i\}$ of $\hh$ under $(, )$. Define $\wrho := \rho + n \Lambda_0$.

Let $q$ be a non-zero complex number with $|q| < 1$.  The quantum affine algebra $U_q(\asl_{n})$ is the Hopf algebra generated as an algebra by $e_i, f_i, q^{\pm h_i}$ for $0 \leq i \leq r$ with relations
\begin{align*}
[q^{h_i}, q^{h_j}] &= 0 \qquad q^{h_i} e_j q^{-h_j} = q^{(h_i, \alpha_j)} e_j \qquad q^{h_i} f_j q^{-h_j} = q^{-(h_i, \alpha_j)} f_j \qquad [e_i, f_j] = \delta_{ij} \frac{q^{h_i} - q^{-h_i}}{q - q^{-1}} \\
\sum_{k = 0}^{1 - a_{ij}}& (-1)^k \binom{1 - a_{ij}}{k}_{q} e_i^{1 - a_{ij} - k} e_j e_i^k = 0 \qquad \sum_{k = 0}^{1 - a_{ij}} (-1)^k \binom{1 - a_{ij}}{k}_{q} f_i^{1 - a_{ij} - k} f_j f_i^k = 0,
\end{align*}
where we use the notations $[n] = \frac{q^n - q^{-n}}{q - q^{-1}}$, $[n]! = [n] \cdots [1]$, and $\binom{a}{b}_q = \frac{[a]_q!}{[b]_q! [a - b]_q!}$.  The coproduct of $U_q(\asl_{n})$ is
\begin{align*}
\Delta(e_i) &= e_i \otimes 1 + q^{h_i} \otimes e_i \qquad \Delta(f_i) = f_i \otimes q^{-h_i} + 1 \otimes f_i \qquad 
\Delta(q^{h_i}) = q^{h_i} \otimes q^{h_i},
\end{align*}
the antipode is
\[
S(e_i) = -  q^{-h_i}e_i \qquad S(f_i) = - f_iq^{h_i} \qquad S(q^{h_i}) = q^{-h_i},
\]
and the counit is 
\[
\eps(e_i) = \eps(f_i) = 0 \qquad \eps(q^{h_i}) = 1.
\]
Let $U_q(\wsl_n)$ be the extension of $U_q(\asl_n)$ by a generator $q^d$ which commutes with $q^{h_i}$ and interacts with $e_i$ and $f_i$ via
\[
[q^d, e_i] = [q^d, f_i] = 0 \text{ for $i \neq 0$} \qquad q^de_0 q^{-d} = q e_0 \qquad q^df_0q^{-d} = q^{-1} f_0
\]
and on which the coproduct, antipode, and counit are
\[
\Delta(q^d) = q^d \otimes q^d \qquad S(q^d) = q^{-d} \qquad \eps(q^d) = 1.
\]

\begin{remark}
This coproduct is the opposite of the one in \cite{FR, ESV} but agrees with those in the bosonizations of \cite{KQS, Kon, Mat}.  Our motivation for using it is to apply results from \cite{Sun:qafftr}, which uses the bosonization and coproduct of \cite{Mat}.  For a presentation of results parallel to those of \cite{ESV} in the coproduct of this paper, we refer the reader to \cite{Sun:esv-mod}.
\end{remark}

\subsection{Affine Macdonald polynomials}

Fix an integer $\mk \geq 0$.  For a dominant integral weight $\mu + k \Lambda_0$, let $L_{\mu + k\Lambda_0}$ denote the irreducible integrable module for $U_q(\asl_{n})$ with highest weight $\mu + k\Lambda_0$ and highest weight vector $v_{\mu + k \Lambda_0}$.  Let $V$ denote the dimension $n$ irreducible fundamental representation of $U_q(\sl_{n})$.  For $v \in \Sym^{n (\mk - 1)}V[0]$, we have a unique intertwiner 
\[
\Upsilon_{\mu, k, \mk}^v(z): L_{\mu + k \Lambda_0 + (\mk - 1)\wrho} \to L_{\mu + k \Lambda_0 + (\mk - 1)\wrho} \hotimes \Sym^{n(\mk - 1)}V(z)
\]
such that $\Upsilon^v_{\mu, k, \mk}(z) v_{\mu + k \Lambda_0 + (\mk - 1)\wrho} = v_{\mu + k \Lambda_0 + (\mk - 1)\wrho} \otimes v + (\text{l.o.t.})$, where $(\text{l.o.t.})$ denotes terms of lower weight in the first tensor factor.

Fix a choice of $w_0 \in \Sym^{n(\mk - 1)}V [0]$ and make the identification $\Sym^{n(\mk - 1)}V [0] \simeq \CC \cdot w_0$.  Define the trace function
\[
\chi_{\mu, k, \mk}(q, \lambda, \omega) = \Tr|_{L_{\mu + k \Lambda_0 + (\mk - 1)\wrho}}\Big(\Upsilon^{w_0}_{\mu, k, \mk}(z) q^{2\lambda + 2\omega d}\Big),
\]
where the trace is independent of $z$.  In \cite{EK3}, the affine Macdonald polynomial for $\asl_n$ at $t = q^{\mk}$ was defined to be
\[
J_{\mu, k, \mk}(q, \lambda, \omega) := \frac{\chi_{\mu, k, \mk}(q, \lambda, \omega)}{\chi_{0, 0, \mk}(q, \lambda, \omega)}.
\]
It is a symmetric Laurent series with highest term $q^{(\mu + k \Lambda_0, 2\lambda + 2\omega d)}$.  The denominator $\chi_{0, 0, \mk}(q, \lambda, \omega)$ was determined in \cite{EK3} up to an unknown multiplicative factor as follows.

\begin{prop}[{\cite[Theorem 11.1]{EK3}}] \label{prop:ek-const}
There is a function $f_{n, \mk}(q, q^{-2\omega})$ with unit constant term whose formal power series expansion in $q^{-2\omega}$ has rational function coefficients in $q$ such that 
\[
\chi_{0, 0, \mk}(q, \lambda, \omega) = f_{n, \mk}(q, q^{-2\omega}) q^{2(\mk - 1)(\rho, \lambda)} \prod_{i = 1}^{\mk - 1} \prod_{\alpha > 0} (1 - q^{-2(\alpha, \lambda + \omega d) + 2i}).
\]
\end{prop}

\begin{remark}
Our function $f_{n, k}(q, q^{-2\omega})$ corresponds to the function $f(p, q)$ in \cite[Theorem 11.1]{EK3}.
\end{remark}

\begin{remark}
To avoid conflict with the use of $k$ for level, we use $\mk$ to denote the parameter of the Macdonald polynomial.  This corresponds to the variable $k$ in \cite{EK3} and $m$ in \cite{FV4}.
\end{remark}

\subsection{Modified affine Macdonald conjectures}

In \cite{EK3}, denominator and evaluation conjectures were stated for affine Macdonald polynomials by analogy with the finite-type setting.  In this section, we make refinements and corrections to these conjectures based on the behavior of the $U_q(\asl_2)$ case determined from computer computations in Magma and Theorems \ref{thm:f-val} and \ref{thm:eval-conj} in Section \ref{sec:aff-ell}.  We begin with a conjecture which refines the result of \cite[Theorem 11.1]{EK3} by specifying a precise form for the affine Macdonald denominator.

\begin{conj}[Affine denominator conjecture] \label{conj:f-val-conj}
The affine denominator is given by 
\[
\chi_{0, 0, \mk}(q, \lambda, \omega) = q^{2(\mk - 1)(\rho, \lambda)} \frac{\prod_{i = 1}^{\mk - 1} (q^{-2\omega + 2i}; q^{-2\omega})}{\prod_{i = 1}^{\mk - 1} (q^{-2\omega + 2ni}; q^{-2\omega})}  \prod_{i = 1}^{\mk - 1} \prod_{\alpha > 0} (1 - q^{-2(\alpha, \lambda +\omega d) + 2i})^{\mult(\alpha)}.
\]
\end{conj}

\begin{remark}
Conjecture \ref{conj:f-val-conj} is equivalent to the fact that the correction term $f_{n, \mk}(q, q^{-2\omega})$ is given by
\[
f_{n, \mk}(q, q^{-2\omega}) = \prod_{i = 1}^{\mk - 1} \prod_{\substack{\alpha > 0 \\ \alpha \text{ imaginary}}} \frac{(1 - q^{- 2(\alpha, \lambda + \omega d) + 2i})^{n-1}}{1 - q^{-2(\alpha, \lambda + \omega d) + 2ni}} = \frac{\prod_{i = 1}^{\mk - 1} (q^{-2\omega + 2i}; q^{-2\omega})^{n-1}}{\prod_{i = 1}^{\mk - 1} (q^{-2\omega + 2ni}; q^{-2\omega})}.
\]
However, recalling that the multiplicity of each imaginary root for $\asl_n$ is $n - 1$, it is more natural to write the conjectured expression as 
\[
\Delta_\mk(q, q^{-2\omega}) \cdot q^{2(\mk - 1) (\rho, \lambda)} \prod_{i = 1}^{k - 1} \prod_{\alpha > 0} (1 - q^{-2(\alpha, \lambda + \omega d) + 2i})^{\mult(\alpha)},
\]
where $\Delta_\mk(q, q^{-2\omega})$ is a correction factor given by
\begin{equation} \label{eq:delta-def}
\Delta_\mk(q, q^{-2\omega}) := \frac{\prod_{i = 1}^{\mk - 1} (q^{-2\omega + 2i}; q^{-2\omega})}{\prod_{i = 1}^{\mk - 1} (q^{-2\omega + 2n i}; q^{-2\omega})}.
\end{equation}
and
\begin{equation} \label{eq:aff-analogue}
q^{2(\mk - 1) (\rho, \lambda)} \prod_{i = 1}^{k - 1} \prod_{\alpha > 0} (1 - q^{-2(\alpha, \lambda + \omega d) + 2i})^{\mult(\alpha)}
\end{equation}
is an affine analogue of the finite-type Macdonald denominator.  Note that the presence of the multiplicity in the exponents of (\ref{eq:aff-analogue}) is motivated by their appearance in the Weyl-Kac character formula.
\end{remark}

\begin{remark}
For $n = 2$ and $\mk = 2$, Conjecture \ref{conj:f-val-conj} is established in Theorem \ref{thm:f-val}.  For $n = 2$ and $3 \leq \mk \leq 15$ and $n = 3$ and $2 \leq \mk \leq 3$, we verified in the computer algebra system Magma (see \cite{Magma}) that Conjecture \ref{conj:f-val-conj} holds up to first order in $q^{-2\omega}$.  Our code and instructions for reproducing our computations are provided at \url{github.com/yi-sun/aff-mac}.
\end{remark}

For the affine evaluation conjecture, we must modify \cite[Conjecture 11.3]{EK3}, which in our notations states that $J_{\mu, k, \mk}(q, \mk \rho, \mk n)$ is given by 
\[
J^\text{conj}_{\mu, k, \mk}(q, \mk\rho, \mk n) := q^{-2(\mu, \mk \rho)} \prod_{\alpha > 0} \prod_{i = 0}^{\mk - 1} \frac{1 - q^{2(\alpha, \mu + k \Lambda_0 + \mk \wrho) + 2i}}{1 - q^{2(\alpha, \mk \wrho) + 2i}}.
\]
Note that the presence of $q^{2\mk n d}$ in the expression for
\[
\chi_{\mu, k, \mk}(q, \mk \rho, \mk n) = \Tr|_{L_{\mu + k \Lambda_0 + (\mk - 1)\wrho}}\Big(\Upsilon^{w_0}_{\mu, k, \mk}(z) q^{2\mk\rho + 2\mk n d}\Big),
\]
implies that the trace can only converge for $|q| > 1$.  On the other hand, the conjectural expression $J^\text{conj}_{\mu, k, \mk}(q, \mk\rho, \mk n)$ of \cite[Conjecture 11.3]{EK3} does not converge for $|q| > 1$, meaning that we must replace it by an expression which does.  We propose in the following Conjecture \ref{conj:aff-eval} a modification which converges for $|q| > 1$ and incorporates a multiplicative correction of form similar to that of Conjecture \ref{conj:f-val-conj}.

\begin{conj}[Affine evaluation conjecture] \label{conj:aff-eval}
For $|q| > 1$, we have that
\begin{equation} \label{eq:aff-eval-eq}
J_{\mu, k, \mk}(q, \mk \rho, \mk n) = q^{2(\mu, \mk \rho)} \frac{\prod_{i = 1}^{\mk - 1} (q^{-2i}; q^{-2(k + \mk n)})}{\prod_{i = 1}^{\mk - 1} (q^{-2ni}; q^{-2(k + \mk n)})} \frac{\prod_{i = 1}^{\mk - 1} (q^{-2ni}; q^{-2\mk n})}{\prod_{i = 1}^{\mk - 1} (q^{-2i}; q^{-2\mk n})}  \prod_{\alpha > 0} \prod_{i = 0}^{\mk - 1} \frac{(1 - q^{-2(\alpha, \mu + k \Lambda_0 + \mk \wrho) - 2i})^{\mult(\alpha)}}{(1 - q^{-2(\alpha, \mk \wrho) - 2i})^{\mult(\alpha)}},
\end{equation}
where
\[
\frac{\prod_{i = 1}^{\mk - 1} (q^{-2i}; q^{-2(k + \mk n)})}{\prod_{i = 1}^{\mk - 1} (q^{-2ni}; q^{-2(k + \mk n)})} = \prod_{i = 1}^{\mk - 1} \prod_{\substack{\alpha > 0 \\ \alpha \text{ imaginary}}} \frac{1 - q^{-2(\alpha, \mu + k \Lambda_0 + \mk \wrho) - 2i}}{1 - q^{-2(\alpha, \mu + k \Lambda_0 + \mk \wrho) - 2ni}}
\]
and
\[
 \frac{\prod_{i = 1}^{\mk - 1} (q^{-2ni}; q^{-2\mk n})}{\prod_{i = 1}^{\mk - 1} (q^{-2i}; q^{-2\mk n})} = \prod_{i = 1}^{\mk - 1} \prod_{\substack{\alpha > 0 \\ \alpha \text{ imaginary}}} \frac{1 - q^{-2(\alpha, \mk \wrho) - 2ni}}{1 - q^{-2(\alpha, \mk \wrho) - 2i}}.
\]
\end{conj}

\begin{remark}
For $n = 2$ and $\mk = 2$, Conjecture \ref{conj:aff-eval} is established in Theorem \ref{thm:eval-conj}.
\end{remark}

\begin{remark}
The conjectured expression (\ref{eq:aff-eval-eq}) in Conjecture \ref{conj:aff-eval} differs from $J^\text{conj}_{\mu, k, \mk}(q, \mk \rho, \mk n)$ in two ways.  First, the ordinary Macdonald evaluation formula of \cite[Corollary 4.4]{EK-mac} may be rewritten as
\[
P_\mu(q^{2\mk \rho}; q^2, q^{2\mk}) = q^{2(\mu, \mk \rho)} \prod_{\alpha > 0} \prod_{i = 0}^{\mk - 1} \frac{1 - q^{-2(\alpha, \mu + \mk \rho) - 2i}}{1 - q^{-2(\alpha, \mk \rho) - 2i}},
\]
where $P_\mu(x; q, t)$ denotes the ordinary Macdonald polynomial.  One multiplicative factor is given by taking the affine analogue of this form, which has better convergence properties for $|q| > 1$ in the affine setting.  Second, the multiplicative correction factor 
\[
\frac{\prod_{i = 1}^{\mk - 1} (q^{-2i}; q^{-2(k + \mk n)})}{\prod_{i = 1}^{\mk - 1} (q^{-2ni}; q^{-2(k + \mk n)})} \frac{\prod_{i = 1}^{\mk - 1} (q^{-2ni}; q^{-2\mk n})}{\prod_{i = 1}^{\mk - 1} (q^{-2i}; q^{-2\mk n})}
\]
has been added in analogy with the presence of $f_{n, \mk}(q, q^{-2\omega})$.
\end{remark}

\subsection{Proof of affine Macdonald conjectures for $n = 2$ and $\mk = 2$} \label{sec:aff-ell}

In the remainder of this section, we resolve Conjectures \ref{conj:f-val-conj} and \ref{conj:aff-eval} for $n = 2$ and $\mk = 2$.  This partially answers two questions from \cite{EK3}.  Our approach uses a connection between affine Macdonald polynomials for $n = 2$ and $\mk = 2$ and elliptic Macdonald polynomials conjectured to exist by Felder-Varchenko in \cite{FV4} and established in a precise form by Sun in \cite{Sun:qafftr}.  To state the relation in Proposition \ref{prop:fv-conj}, we specialize and make explicit Proposition \ref{prop:ek-const} in Proposition \ref{prop:ek-const-2}.  

\begin{prop}[{\cite[Theorem 11.1]{EK3}}] \label{prop:ek-const-2}
There is a function $f_{2, 2}(q, q^{-2\omega})$ with unit constant term whose formal power series expansion in $q^{-2\omega}$ has rational function coefficients in $q$ such that 
\[
\chi_{0, 0, 2}(q, \lambda, \omega) = f_{2, 2}(q, q^{-2\omega}) q^\lambda (q^{-2\lambda + 2}; q^{-2\omega})(q^{2\lambda + 2} q^{-2\omega}; q^{-2\omega}) (q^{-2\omega + 2}; q^{-2\omega}).
\]
\end{prop}

\begin{prop}[{\cite[Theorem 9.9]{Sun:qafftr}}] \label{prop:fv-conj}
For $|q| > 1$, $|q^{-2\omega}| < |q^{-6}|$, and $q^{-2\mu}$ sufficiently close to $0$, the elliptic and affine Macdonald polynomials for $U_q(\asl_2)$ are related by
\begin{multline*}
J_{\mu, k, 2}(q, \lambda, \omega)\\
 = \frac{P_{\mu, \wk}(2\eta\lambda; -2\eta\omega, \eta)}{2\pi f_{2, 2}(q, q^{-2\omega})} \frac{(q^{-4}; q^{-2\omega})(q^{-2\omega}; q^{-2\omega})^3}{(q^{-2\omega + 2}; q^{-2\omega})} \frac{(q^{-2\omega + 2}; q^{-2\omega}, q^{-2\wk})^2}{(q^{-2\omega - 2}; q^{-2\omega}, q^{-2\wk})^2} \frac{q^{\mu + 4}(q^{-2\mu - 6}; q^{-2\wk})(q^{2\mu + 2} q^{-2\wk}; q^{-2\wk})}{(q^{-4}; q^{-2\wk})(q^{-2\wk}; q^{-2\wk})},
\end{multline*}
where $f_{2, 2}(q, q^{-2\omega})$ is the normalizing function of Proposition \ref{prop:ek-const-2} and $\wk = k + 4$.
\end{prop}

\begin{remark}
The version of \cite[Theorem 9.9]{Sun:qafftr} in Proposition \ref{prop:fv-conj} differs from the published version in the following ways:
\begin{itemize}
\item there is a normalization error in the published version yielding a multiplicative factor of $2\pi$;

\item our function $f_{2, 2}(q, q^{-2\omega})$ corresponds to the function $f(p, q)$ defined in \cite[Theorem 11.1]{EK3} and is related to the function of \cite[Theorem 9.9]{Sun:qafftr} by $f^{\text{\cite{Sun:qafftr}}}(q, q^{-2\omega}) = f_{2, 2}(q, q^{-2\omega}) (q^{-2\omega + 2}; q^{-2\omega})$;

\item our notation for elliptic Macdonald polynomials is related to that of \cite{Sun:qafftr} by $\wJ_{\mu, \kappa}(q, \lambda, \omega) = P_{\mu, \kappa}(2\eta\lambda; -2\eta \omega, \eta)$, where $q = e^{2\pi i \eta}$;

\item the published condition that $|q^{-2\omega}|$ is sufficiently close to $0$ can be weakened to $|q^{-2\omega}| < |q^{-6}|$, since the cycle in Proposition \ref{prop:htf-prop} may be deformed so that the integrand has no poles with $\Imm(\tau) > -5\Imm(\eta)$, meaning that the expression on the right of Proposition \ref{prop:fv-conj} admits a convergent expansion as a series in $q^{-2\omega}$ with no poles in the region $\{|q^{-2\omega}| < |q^{-6}|\}$.
\end{itemize}
\end{remark}

We now use this connection to determine the multiplicative factor $f_{2, 2}(q, q^{-2\omega})$ in Theorem \ref{thm:f-val} and to establish the first case of the affine Macdonald evaluation conjecture in Theorem \ref{thm:eval-conj}.

\begin{thm} \label{thm:f-val}
For $n = 2$ and $\mk = 2$, the function $f_{2, 2}(q, q^{-2\omega})$ takes the value
\[
f_{2, 2}(q, q^{-2\omega}) = \frac{(q^{-2\omega + 2}; q^{-2\omega})}{(q^{-2\omega + 4}, q^{-2\omega})},
\]
meaning that the affine denominator is given by
\[
\chi_{0, 0, 2}(q, \lambda, \omega) = q^\lambda \frac{(q^{-2\omega + 2}; q^{-2\omega})}{(q^{-2\omega + 4}, q^{-2\omega})} (q^{-2\lambda + 2}; q^{-2\omega}) (q^{2\lambda + 2} q^{-2\omega}; q^{-2\omega}) (q^{-2\omega + 2}; q^{-2\omega}).
\]
\end{thm}
\begin{proof}
By definition, we have that $J_{0, 0, 2}(q, \lambda, \omega) = 1$, so by Proposition \ref{prop:fv-conj} and Theorem \ref{thm:ellmac-val}, we have that 
\begin{align*}
f_{2, 2}(q, q^{-2\omega}) &= \frac{P_{0, 4}(2\eta\lambda; -2\eta\omega, \eta)}{2\pi} \frac{(q^{-4}; q^{-2\omega})(q^{-2\omega}; q^{-2\omega})^3}{(q^{-2\omega + 2}; q^{-2\omega})} \frac{(q^{-2\omega + 2}; q^{-2\omega}, q^{-8})^2}{(q^{-2\omega - 2}; q^{-2\omega}, q^{-8})^2} \frac{q^{4}(q^{- 6}; q^{-8})^2}{(q^{-4}; q^{-8})(q^{-8}; q^{-8})}\\
&= - \frac{\Gamma(q^{-6}; q^{-2\omega}, q^{-8})}{\Gamma(q^{-2}; q^{-2\omega}, q^{-8})} \frac{(q^{-2\omega + 2}; q^{-2\omega}, q^{-8})^2}{(q^{-2\omega - 2}; q^{-2\omega}, q^{-8})^2} \frac{(q^{-4}; q^{-2\omega})}{\theta_0(q^4; q^{-2\omega}) (q^{-2\omega + 2}; q^{-2\omega})} \frac{q^4 (q^{-6}; q^{-8})^2 (q^{-4}; q^{-4})}{(q^{-4}; q^{-8})(q^{-8}; q^{-8}) (q^{-2}; q^{-4})}.
\end{align*}
Notice now that 
\begin{align*}
\frac{\Gamma(q^{-6}; q^{-2\omega}, q^{-8})}{\Gamma(q^{-2}; q^{-2\omega}, q^{-8})} \frac{(q^{-2\omega + 2}; q^{-2\omega}, q^{-8})^2}{(q^{-2\omega - 2}; q^{-2\omega}, q^{-8})^2} &= \frac{(q^{-2}; q^{-2\omega}, q^{-8}) (q^{-2\omega - 2}; q^{-2\omega}, q^{-8})}{(q^{-6}; q^{-2\omega}, q^{-8})(q^{-2\omega - 6}; q^{-2\omega}, q^{-8})} \frac{(q^{-2\omega + 2}; q^{-2\omega}, q^{-8})^2}{(q^{-2\omega - 2}; q^{-2\omega}, q^{-8})^2}\\
&= \frac{(q^{-2}; q^{-8}) (q^{-2\omega + 2}; q^{-2\omega})^2}{(q^{-6}; q^{-8})}.
\end{align*}
Substituting back in, we obtain
\[
f_{2, 2}(q, q^{-2\omega}) = -\frac{(q^{-2\omega + 2}; q^{-2\omega})^2(q^{-4}; q^{-2\omega})}{(q^{-2\omega + 2}; q^{-2\omega})\theta_0(q^4; q^{-2\omega})} \frac{q^4 (q^{-6}; q^{-8})^2 (q^{-2}; q^{-8})}{(q^{-2}; q^{-4}) (q^{-6}; q^{-8})} = \frac{(q^{-2\omega + 2}; q^{-2\omega})}{(q^{-2\omega + 4}; q^{-2\omega})}. \qedhere
\]
\end{proof}

\begin{thm} \label{thm:eval-conj}
For $|q| > 1$ and $n = 2$, the affine Macdonald polynomial satisfies the evaluation
\[
J_{\mu, k, 2}(q, 2, 4) = q^{2\mu} \frac{(q^{-2}; q^{-2\wk})}{(q^{-4}; q^{-2\wk})} \frac{\theta_0(q^{-2\mu - 4}; q^{-2\wk}) (q^{-2\mu - 6}; q^{-2\wk})(q^{2\mu + 2} q^{-2\wk}; q^{-2\wk})(q^{-2\wk}; q^{-2\wk}) (q^{-2\wk - 2}; q^{-2\wk})}{(q^{-4}; q^{-2})(q^{-6};q^{-8}) (q^{-2}; q^{-8})}.
\]
\end{thm}
\begin{proof}
Noting that $|q^{-8}| < |q^{-6}|$ for $|q| > 1$, we may apply Proposition \ref{prop:fv-conj} and Theorem \ref{thm:f-val} to find that
\[
J_{\mu, k, 2}(q, 2, 4) = \frac{P_{\mu, \wk}(4\eta; -8\eta, \eta)}{2\pi} \frac{\theta_0(q^{-4}; q^{-8})(q^{-8}; q^{-8})^3}{(q^{-6}; q^{-8})^2} \frac{(q^{-6}; q^{-8}, q^{-2\wk})^2}{(q^{-10}; q^{-8}, q^{-2\wk})^2}\frac{q^{\mu + 4}(q^{-2\mu - 6}; q^{-2\wk})(q^{2\mu + 2} q^{-2\wk}; q^{-2\wk})}{(q^{-4}; q^{-2\wk})(q^{-2\wk}; q^{-2\wk})}.
\]
In multiplicative notation, the result of Theorem \ref{thm:ellmac-eval} states that
\[
P_{\mu, \wk}(4\eta; -8\eta, \eta) = -2 \pi q^{-\mu - 8} \frac{\Gamma(q^{-6}; q^{-2\wk}, q^{-8})}{\Gamma(q^{-2}; q^{-2\wk}, q^{-8})} \frac{\theta_0(q^{2\mu + 4}; q^{-2\wk}) (q^{-2\wk}; q^{-2\wk})^2}{(q^{-8}; q^{-8})(q^{-4}; q^{-4})^2 (q^{-2}; q^{-2})}.
\]
Substituting in, we find that 
\begin{align*}
J_{\mu, k, 2}(q, 2, 4) &= - q^{- 4} \theta_0(q^{2\mu + 4}; q^{-2\wk}) (q^{-2\mu - 6}; q^{-2\wk})(q^{2\mu + 2} q^{-2\wk}; q^{-2\wk}) \frac{\Gamma(q^{-6}; q^{-2\wk}, q^{-8})}{\Gamma(q^{-2}; q^{-2\wk}, q^{-8})}\frac{(q^{-6}; q^{-8}, q^{-2\wk})^2}{(q^{-10}; q^{-8}, q^{-2\wk})^2}\\
&\phantom{==} \frac{(q^{-2\wk}; q^{-2\wk})}{(q^{-4}; q^{-4})^2 (q^{-2}; q^{-2})} \frac{\theta_0(q^{-4}; q^{-8})(q^{-8}; q^{-8})^2}{(q^{-6}; q^{-8})^2(q^{-4}; q^{-2\wk})}.
\end{align*}
Observe now that 
\begin{align*}
\frac{\Gamma(q^{-6}; q^{-2\wk}, q^{-8})}{\Gamma(q^{-2}; q^{-2\wk}, q^{-8})}\frac{(q^{-6}; q^{-8}, q^{-2\wk})^2}{(q^{-10}; q^{-8}, q^{-2\wk})^2} &= \frac{(q^{-6}; q^{-2\wk}, q^{-8}) (q^{-2\wk - 2}; q^{-2\wk}, q^{-8}) (q^{-2}; q^{-2\wk}, q^{-8})}{(q^{-2\wk - 6}; q^{-2\wk}, q^{-8})(q^{-10}; q^{-2\wk}, q^{-8})^2}\\
&= (q^{-6}; q^{-8}) \frac{(q^{-2}; q^{-2\wk})^2}{(q^{-2}; q^{-8})},
\end{align*}
which implies that 
\begin{multline*}
J_{\mu, k, 2}(q, 2, 4) = q^{2\mu}\frac{(q^{-2}; q^{-2\wk})}{(q^{-4}; q^{-2\wk})}\\ \frac{\theta_0(q^{-2\mu - 4}; q^{-2\wk}) (q^{-2\mu - 6}; q^{-2\wk})(q^{2\mu + 2} q^{-2\wk}; q^{-2\wk})(q^{-2\wk}; q^{-2\wk}) (q^{-2\wk - 2}; q^{-2\wk})}{(q^{-4}; q^{-2})(q^{-6};q^{-8}) (q^{-2}; q^{-8})}. \qedhere
\end{multline*}
\end{proof}

\subsection{Limits of the conjectures} \label{sec:conj-limits}

In this section we discuss three degenerations of our affine Macdonald polynomials and the known or conjectured relations of our conjectures to results for other known objects.  Affine Macdonald theory has a four-dimensional space of parameters, with coordinates given by $q$, $t = q^\mk$, $k$, and $\omega$, and each degeneration will be to an integrable system with a three-dimensional space of parameters.  In particular, in the classical, affine Hall, and critical limits, we recover relations to the affine Jack, affine Hall, and elliptic Macdonald-Ruijsenaars integrable systems.

\subsubsection{The classical limit}

In the classical limit $q \to 1$, the affine Macdonald polynomials become the affine Jack polynomials defined in \cite{EK3} in terms of traces of intertwiners between irreducible integrable modules for the classical affine algebra.  More precisely, let $L^\cl_{\mu + k\Lambda_0}$ denote the irreducible integrable module for $U_q(\asl_n)$ with highest weight $\mu + k \Lambda_0$ and highest weight vector $v^\cl_{\mu + k \Lambda_0}$, and let $V^\cl$ denote the dimension $n$ irreducible fundamental representation.  For $v \in \Sym^{n(\mk - 1)} V [0]$, there is a unique classical $U_q(\asl_n)$-intertwiner
\[
\Upsilon^{v, \cl}_{\mu, k, \mk}(z): L^\cl_{\mu + k \Lambda_0 + (\mk - 1) \wrho} \to L^\cl_{\mu + k\Lambda_0 + (\mk - 1)\wrho} \hotimes \Sym^{n (\mk - 1)}V (z)
\]
such that $\Upsilon^{v, \cl}_{\mu, k, \mk}(z) v^\cl_{\mu + k \Lambda_0 + (\mk - 1)\wrho} = v^\cl_{\mu + k \Lambda_0 + (\mk -1 )\wrho} \otimes v + (\text{l.o.t.})$.  Fixing a choice of $w_0 \in \Sym^{n(\mk - 1)}V[0]$ and identifying $\Sym^{n(\mk - 1)}V[0]\simeq \CC \cdot w_0$, we define the classical trace function by
\[
\chi^\cl_{\mu, k, \mk}(\Lambda, \Omega) := \Tr|_{L^\cl_{\mu + k \Lambda_0 + (\mk - 1)\wrho}}\Big(\Upsilon^{w_0, \cl}_{\mu, k, \mk}(z) e^{2\Lambda + 2\Omega d}\Big),
\]
where the trace is independent of $z$.  In \cite{EK3}, the affine Jack polynomial with parameter $\mk$ was defined by
\[
J^\cl_{\mu, k, \mk}(\Lambda, \Omega) := \frac{\chi^\cl_{\mu, k, \mk}(\Lambda, \Omega)}{\chi^\cl_{0, 0, \mk}(\Lambda, \Omega)}.
\]
It is the classical limit of the affine Macdonald polynomial in the sense that
\[
J^\cl_{\mu, k, \mk}(\Lambda, \Omega) = \lim_{\eps \to 0} J_{\mu, k, \mk}(e^\eps, \eps^{-1} \Lambda, \eps^{-1} \Omega).
\]
This limit of Conjecture \ref{conj:f-val-conj} was proven in \cite[Theorem 7.5]{EK3}, as the correction term $f_{n, \mk}(q, q^{-2\omega})$ becomes $1$ in the limit.  For Conjecture \ref{conj:f-val-conj}, the corresponding classical limit is more singular.  The conjectured value in Conjecture \ref{conj:aff-eval} has a singular limit as $q \to 1$, and the conjecture implies that the asymptotics of this limit should match the asymptotics of the evaluation $J^\cl_{\mu, k, \mk}(\eta \cdot \mk \rho, \eta \cdot \mk n)$ of the affine Jack polynomial as $\eta \to 0$.  In terms of the integral formulas for the affine Jack polynomials from \cite[Theorem 3.2]{FSV}, this becomes a statement about asymptotics of integrals.  In the case $n = 2$ and $\mk = 2$, we believe this can be verified using the method of steepest descent.

\subsubsection{The affine Hall limit}

In this section we discuss a limit which is consistent with the form of the conjectured affine denominator in Conjecture \ref{conj:f-val-conj}.  In the context of affine root systems, Macdonald studied a factor similar to $\Delta_{\mk}(q, q^{-2\omega})$ from (\ref{eq:delta-def}) in \cite{Mac03}; more precisely, he proved the following identity.

\begin{prop}[{\cite{Mac03}}] \label{prop:mac-identity}
If $W_\aff$ is the affine Weyl group of $\asl_n$, we have that 
\[
\Delta^\text{Mac}(t, p) := \frac{(p^2t^{2n}; p^2)}{(p^2t^2; p^2)} = \frac{1}{W_\aff(t^{2})} \sum_{w \in W_\aff} w \cdot \left(\frac{\prod_{\alpha > 0} (1 - t^{2(\alpha, \lambda) + 2} p^{2(\alpha, d)})}{\prod_{\alpha > 0} (1 - t^{2(\alpha, \lambda)} p^{2(\alpha, d)})}\right)^{\mult(\alpha)},
\]
where
\[
W_\aff(t) := \sum_{w \in W_\aff} t^{\ell(w)}.
\]
\end{prop}

The factor $\Delta^\text{Mac}(t, p)$ has recently appeared as a normalization factor in the affine Satake isomorphism and affine Gindikin-Karpelevich formula in the works \cite{BFK12, BK13, BGKP14, BKP16} of Braverman, Finkelberg, Garland, Kazhdan, and Patnaik.  These works involved the study of affine Kac-Moody groups and should correspond to the limit of affine Macdonald polynomials to affine Hall-Littlewood polynomials.  By the general philosophy of Macdonald theory, the resulting expressions are therefore expected to be related to the limit of our affine Macdonald theory as $q \to 0$ and $q^\mk$ and $q^{-2\omega}$ are held constant.  Indeed, notice that for
\[
\Delta(q, t, p) := \frac{(p^2 q^2; p^2, q^{2})}{(p^2 t^2; p^2, q^{2})} \frac{(p^2 t^{2n}; p^2, q^{2n})}{(p^2 q^{2n}; p^2, q^{2n})},
\]
we have that $\Delta_\mk(q, q^{-2\omega}) = \Delta(q, q^\mk, q^{-\omega})$.  The desired limit transition is therefore
\[
\lim_{q \to 0} \Delta(q, t, p) = \frac{(p^2 t^{2n}; p^2)}{(p^2 t^2; p^2)} = \Delta^\text{Mac}(t, p),
\]
providing a link between our correction factor and the affine Hall case.

\subsubsection{The critical limit}

Noting that the dual Coxeter number for $U_q(\asl_n)$ is given by $n$, the critical limit corresponds to $\kappa = k + \mk n \to 0$; in particular, if $\mk = 1$, we recover the critical level $k \to -n$. In the classical limit $q \to 1$, the $q$-KZB heat equation becomes the classical KZB heat equation, which at the critical level reduces to the elliptic Calogero-Sutherland integrable system.  It was shown in \cite{EK2} that this limit preserves eigenfunctions in the $\asl_2$ case, meaning that the relevant limit transition sends trace functions for Verma modules for $\asl_2$ to eigenfunctions of the classical Lam\'e operator.  

In the full $q$-deformed setting, we expect the $q$-KZB heat equation to degenerate to the elliptic Macdonald-Ruijsenaars integrable system.  We may also check that in the case of $U_q(\asl_2)$, the critical limit preserves eigenfunctions.  That is, by applying the steepest descent method to the integral formula for Felder-Varchenko functions, we find that the limit transition sends trace functions for Verma modules for $U_q(\asl_2)$ to the eigenfunctions for the $q$-Lam\'e integrable system obtained via algebraic Bethe ansatz by Felder-Varchenko in \cite{FV5, FV-ruij}.

In both the classical and $q$-deformed settings, the critical limit becomes significantly more complicated in the symmetrized setting of affine Jack and affine Macdonald polynomials, and we have not yet been able to identify the limiting object in terms of the elliptic Calogero and Macdonald-Ruijsenaars integrable systems.  Once such a connection is made, the limit of Conjecture \ref{conj:aff-eval} would be a special value identity for certain eigenfunctions related to these systems.

\bibliographystyle{alpha}
\bibliography{spec-val-bib}
\end{document}